\documentclass{amsart}
\newtheorem{thm}{Theorem}[section]
\newtheorem{cor}[thm]{Corollary}
\newtheorem{lem}[thm]{Lemma}
\newtheorem{prop}[thm]{Proposition}
\theoremstyle{definition}
\newtheorem{defn}[thm]{Definition}
\theoremstyle{remark}
\newtheorem{rem}[thm]{Remark}
\newtheorem{example}[thm]{Example}
\numberwithin{equation}{section}
\usepackage{graphicx}
\usepackage{amsmath,amsfonts,amssymb}
\usepackage{enumerate}

\newcommand{\R}{\mathbb R}
\newcommand{\Z}{\mathbb Z}

\newcommand{\To}{\longrightarrow}

\begin{document}

\title[Singularities of discrete improper indefinite affine spheres]{Singularities of discrete improper indefinite affine spheres}%
\author{Anderson Reis de Vargas and Marcos Craizer}%
\email{}%

\thanks{Author's e-mail address: anderson.vargas.1@cp2.edu.br and craizer@puc-rio.br\\
Anderson R. Vargas is a teacher in Colégio Pedro II. Marcos Craizer is a professor in Pontifical Catholic University of Rio de Janeiro.}%

\begin{abstract}
In this paper we consider  discrete improper affine spheres based on asymptotic nets. In this context, we distinguish the discrete edges and vertices that must be considered singular. The singular edges can be considered as discrete cuspidal edges, while some of the singular vertices can be considered as discrete swallowtails. The classification of singularities of discrete nets is quite a difficult task, and our results can be seen as a first step in this direction. We also prove some characterizations of ruled discrete improper affine spheres which are analogous to the smooth case.
\end{abstract}
\maketitle

\section{Introduction}

In this paper we consider asymptotic nets, which are natural nets for the discretization
of surfaces parameterized by asymptotic coordinates. Asymptotic nets have been the object of
recent as well as ancient research by many geometers, as one can see in the list of references of
this paper (\cite{Bobenko2008}, \cite{Craizer2008}, \cite{Craizer2010}, \cite{Craizer2011}, \cite{Rorig2014}, \cite{Kaferbock2013}, \cite{Matsuura2003}, \cite{Milan2014})

There are many classes of affine surfaces with indefinite metric that have been defined as subclasses of asymptotic nets:
Bobenko and Schief have defined discrete affine spheres \cite{Bobenko1999}, Matsuura and Urakawa have considered discrete improper
affine spheres \cite{Matsuura2003}, and, generalizing this latter concept, Craizer, Anciaux and Lewiner have defined discrete affine minimal surfaces \cite{Craizer2008}. In this paper we shall consider the singularities of the discrete improper indefinite affine spheres (DIIAS). We shall also describe some results concerning the ruled case.

Smooth improper affine spheres with indefinite metric (IIAS) can be obtained from a pair of planar curves by the so called centre-chord construction. The generic singularities of an IIAS are cuspidal edges and swallowtails (\cite{Ishikawa2006}), and the projection of the singularities in the plane of the curves define a new planar curve called midpoint tangent locus (MPTL) (\cite{Giblin2008}), which consists of the midpoints of the chords connecting points of both curves with parallel tangents. Moreover, the projection of the cuspidal edges of the IIAS are smooth points of the MPTL, while the projection of the swallowtails of the IIAS are cusps of the MPTL (\cite{Craizer2011}).

In this paper we consider the discrete analog of the centre-chord construction (\cite{Kobayashi2020}) and define the corresponding discrete midpoint tangent locus (DMPTL). The DMPTL is a polygonal line and some special vertices can be regarded as discrete cusps. Then we define singular edges and vertices of the DIIAS as those who project in the MPTL. All singular edges of the DIIAS will be considered as discrete cuspidal edges, while those vertices which project into cusps of the MPTL will be regarded as discrete swallowtails. With this approach, we can give a simple and easily verifiable definition for discrete cuspidal edges and discrete swallowtails in an arbitrary asymptotic net. The consistency of these definitions is remarkable, since it is quite difficult to have good definitions of singularities in the discrete nets.
For a discussion of this question, see \cite{Rossman2018}. For other classes of discrete surfaces with singularities, see \cite{Hoffmann2012} and \cite{Yasumoto2015}.

In the smooth case, a ruled improper affine sphere (IIAS) with indefinite Blaschke metric is affinely congruent to the graph of $z(x,y)= xy + \phi(x)$, for some smooth function $\phi$ (\cite{Nomizu1994}). Among these ruled surfaces, we can distinguish the Cayley surface, defined by $\phi(x)=-\frac{x^3}{3}$. It is characterized by the conditions $C\neq 0$ and $\nabla C=0$, where $C$ is the cubic form and $\nabla $ the induced connection. Finally, the generic singularities of ruled IIAS can be shown as just cuspidal edges. In this paper we have obtained analogous results for the discrete improper affine spheres.

The paper is organized as follows: In Section 2, we give all definitions and results concerning smooth IIAS that are relevant to the discrete setting. In Section 3, we discuss DIIAS, emphasizing the centre-chord construction. Section 4 is the core of the paper, where we define discrete cuspidal edges and discrete swallowtails. In Section 5 we discuss ruled DIIAS.

\section{Preliminaries on smooth affine theory}

\subsection{Affine differential structure}

Consider a parameterized smooth surface $q:U\subset\R^2\To\R^3$, where $U$ is an open subset of the plane and, for $(u,v)\in U$, let
$$
L(u,v)=[\,q_u,q_v,q_{uu}], \ M(u,v)=[\,q_u,q_v,q_{uv}], \ N(u,v)=[\,q_u,q_v,q_{vv}],
$$
where $f_u (f_v)$ denotes the partial derivative of a function $f$ with respect to $u (v)$, and $[\cdot,\cdot,\cdot]$ denotes the determinant.

The surface $q$ is said to be non-degenerate if $LN-M^2\neq0$, and in this case, the Berwald-Blaschke metric is defined by
$ds^2=\frac{1}{|LN-M^2|^{1/4}}\left(L\,du^2+2M\,du\,dv+N\,dv^2\right).$ If $LN-M^2>0$, the metric is called \emph{definite} and the surface is locally convex. On the contrary, if $LN-M^2<0$, the metric is called \emph{indefinite} and the surface is locally hyperbolic, \emph{i.e.}, the tangent plane crosses the surface.

From now on we shall assume that the affine surface has indefinite metric. We may assume that $(u,v)$ are \emph{asymptotic} coordinates, i.e.,  $L=N=0$. In this case, it is possible to take $M>0$, and the affine Blaschke metric takes the form $ds^2=2\Omega\,du\,dv$, where $\Omega^2=M$. Without loss of generality, we shall take $\Omega>0$.

In asymptotic coordinates the structural equations are given by
\begin{equation}\label{Gaussequation}
  \begin{array}{l}
    q_{uu}= \dfrac{\Omega_u}{\Omega}\,q_u+\dfrac{A}{\Omega}\,q_v \phantom{\dfrac{\frac{1}{1}}{\frac{1}{1}}}\\
    q_{vv}= \dfrac{B}{\Omega}\,q_u+\dfrac{\Omega_v}{\Omega}\,q_v \phantom{\dfrac{\frac{1}{1}}{\frac{1}{1}}}\\
    q_{uv}= \Omega\,\xi
  \end{array}
\end{equation}
where $\Omega=[q_u,q_v,\xi]$, $A=[q_u,q_{uu},\xi]$ and $B=[q_v,q_{vv},\xi]$ are the coefficients of the affine cubic form $A\,du^3+B\,dv^3$, and $\xi$ is the \emph{affine normal} vector field (\cite{Buchin1983}). The surface $q$ is called an {\it improper affine sphere}
(IIAS) if $\xi$ is constant.

From now on we shall be interested only in IIAS. In this case, the compatibility equations
are given by
\begin{equation}\label{Compatibility-eq-smooth}
  \Omega_{uv}\Omega-\Omega_u\Omega_v=AB,\,\,A_v=0\,\,\textrm{and}\,\,B_u=0.
\end{equation}

\subsection{The centre-chord construction}

Consider two smooth planar curves $\alpha:I\To\R^2$ and $\beta:J\To\R^2$, where $I,J\subset\R$. Denote
$$
x(u,v)=\frac{1}{2}(\alpha(u)+\beta(v)), \ \ y(u,v)=\frac{1}{2}(\beta(v)-\alpha(u)),
$$
and define $z(u,v)$ by the relations $z_u=[x_u,y]$ and $z_v=[x_v,y]$. The following proposition is proved in  \cite{Craizer2011}.

\begin{prop}
The map $q:I\times J\To\R^3$ given by $q(s,t)\To(x(s,t),z(s,t))$ is an IIAS, and conversely, any IIAS can be obtained from a pair of smooth planar curves $(\alpha,\beta)$ by the above construction. Moreover,
\begin{equation}\label{OmegaCentreChord}
\Omega(u,v)=\tfrac{1}{4}\left[ \tfrac{d\alpha}{du}, \tfrac{d\beta}{dv} \right], \ A=\tfrac{1}{4}[\tfrac{d\alpha}{du},\tfrac{d^2\alpha}{du^2}],\ B=-\tfrac{1}{4}[\tfrac{d\beta}{dv},\tfrac{d^2\beta}{dv^2}].
\end{equation}
\end{prop}

\subsection{Singularities of the IIAS and the MPTL}\label{sec:CuspsMPTL}

Even if $\Omega$ changes its sign, we shall call the surface obtained by the centre-chord construction an IIAS. In this context, the IIAS may present singularities.

The singular set $S$ of $q$ consists of all pairs $(u,v)$ for which $\Omega(u,v)=0$. From Equation \eqref{OmegaCentreChord}, it follows that $(u,v)\in S$ if and only if $\tfrac{d\alpha}{du}$ and $\tfrac{d\beta}{dv}$ are parallel. The set $x(S)$ is called the \emph{midpoint parallel tangent locus} (MPTL) of the pair $(\alpha,\beta)$ (\cite{Giblin2008}).
Generically, the MPTL is a smooth regular curve with some cusps (\cite{Craizer2011}). Moreover, a point $(u_0,v_0)\in S$ is a cusp if and only if, in a neighborhood of $(u_0,v_0)$, the MPTL is contained in a half-plane determined by the line supporting the chord $\alpha(u_0)\beta(v_0)$. The following proposition holds for a generic IIAS (\cite{Craizer2011}):

\begin{prop}\label{prop-germes}$\phantom{.}$ Let $(u_0,v_0)\in S$.
\begin{enumerate}[(i)]
  \item  The singular point $(u_0,v_0)$ is a smooth point of the MPTL if and only if it is a cuspidal edge of the IIAS.
  \item The singular point $(u_0,v_0)$ is an ordinary cusp of the MPTL if and only if it is a swallowtail of the IIAS.
\end{enumerate}
\end{prop}

\subsection{Ruled improper indefinite affine spheres}

The surface $q$ is a ruled surface if either $u$-curves or $v$-curves are all straight lines.
From the structural equations \eqref{Gaussequation}, it is clear that an IIAS is ruled if and only if $A=0$ or $B=0$.

Assume that $B=0$. Then with a change of variable of the form $V=V(v)$ we may also assume that $q_{vv}=0$, which implies that $\Omega$ is in fact a function of $u$, independent of $v$.  Now a change of variable $U=\Omega(u)$ implies that we can in fact assume that $\Omega$ is a constant, say $1$. For a ruled IIAS, we shall call such a parameterization {\it normalized}. In a normalized parameterization, one can easily verify that the cubic form is given by $C=-2Adu^3$ and so $\nabla C=0$ if and only if $A_u=0$, where $\nabla$ denotes the induced connection.

The following result is proved in \cite{Nomizu1994}:

\begin{thm}\label{teo-ruled-improper-affine-sphere}
If $q$ is a smooth ruled IIAS, then it is locally of the form $z=x^1x^2+\varphi(x^1)$ where $\varphi$ is an arbitrary function of $x^1$.
\end{thm}

One important example of such a surface is the so called Cayley surface, when $\varphi(x^1)=-\frac{(x^1)^3}{3}$. The Cayley surface can be parameterized in asymptotic coordinates by
\begin{equation}\label{Cayley-surface-asymptotic}
  q(u,v)=\left(u,v+\frac{au^2}{2},uv+\frac{au^3}{6}\right)
\end{equation}
where $(u,v)\in U\subset\R^2$ and $a\neq 0$ is a constant. Note that this parameterization is normalized and $A=a$, which implies that $\nabla C=0$.
Next theorem implies that the conditions $C\neq 0$ and $\nabla C=0$ state that a ruled IIAS is in fact affinely equivalent to the Cayley surface:

\begin{thm}\label{Cayley-surface-characterization}
Let $q(u,v)$ be a normalized parameterization of a ruled IIAS. Then $q(u,v)$ is affinely congruent to a Cayley surface if and only if $A\neq 0$ and $A_u=0$.
\end{thm}

\begin{proof}
Assume that $A$ is a non-zero constant $a$. Then Equations \eqref{Gaussequation} imply that
\begin{equation}\label{StructuralCayleyNormal}
q_{uu}=aq_v,\ \ q_{uv}=(0,0,1), \ \ q_{vv}=0.
\end{equation}
By an affine transformation, we may assume that $q(0,0)=(0,0,0)$, $q_u(0,0)=(1,0,0)$ and $q_v(0,0)=(0,1,0)$. Now Equations \eqref{StructuralCayleyNormal} show that $q_v=(0,1,u)$, and $q_u=(1,au,\tfrac{au^2}{2}+v)$, which in turn implies that $q$ is given by Equation \eqref{Cayley-surface-asymptotic}.
\end{proof}

Any ruled IIAS can be otained by the centre-chord construction from a planar curve $\alpha(u)$ and a planar line $\beta(v)$. The singular points of the IIAS are the pairs $(u_0,v)$ such that $\tfrac{d\alpha}{du}(u_0)$ is parallel to $\beta$ and $v\in\mathbb{R}$. Thus the MPTL is generically a discrete set of lines parallel to $\beta$ and the singularities of the IIAS is a discrete set of spatial lines, whose points are all cuspidal edges of the IIAS.

\section{Basics on asymptotic nets}

\subsection{Discrete derivatives}

Given a discrete function $f: I\subset\Z\To\R^l$, its discrete derivative
is defined as
$f'(u+\tfrac{1}{2})=f(u+1)-f(u)$. Similarly, for $g: I^*\subset\Z^*\to\R^l$,  its discrete derivative is defined by
$g'(u)=g(u+\tfrac{1}{2})-g(u-\tfrac{1}{2})$. We are denoting by $\Z^*=\Z+\{\tfrac{1}{2}\}$ the dual of $\Z$.

Consider now $f: D\subset\Z^2\To\R^l$. The first discrete partial derivative $f_1$ at $(u,v)\in D$ is the discrete derivative of $f(u,v)$ with respect to $u$, with fixed $v$. In other words,
$$
f_1(u+\tfrac{1}{2},v)=f(u+1,v)-f(u,v).
$$
The second discrete partial derivative $f_2$ is defined similarly. In the same way we can  define the discrete partial derivatives of a function $g: D^*\subset(\Z^2)^*\To\R^l$, where $(\Z^2)^*=\Z^2+\{(\tfrac{1}{2},\tfrac{1}{2})\}$ denotes the dual of $\Z^2$.

The discrete second order partial derivatives of $f$, denoted $f_{ij}$, are defined to be the discrete partial derivative of $f_i$ in the $j$-direction, $i,j=1,2$.
Observe that
$$
f_{12}(u+\tfrac{1}{2},v+\tfrac{1}{2})=f_{21}(u+\tfrac{1}{2},v+\tfrac{1}{2})=f(u+1,v+1)-f(u,v+1)-f(u+1,v)+fu,v).
$$
Similarly, we can define discrete second order partial derivatives for a function $g: D^*\subset(\Z^2)^*\To\R^k$ and verify that $g_{12}=g_{21}$.

Finally, discrete third order partial derivatives $f_{ijk}$ are defined to be the discrete partial derivative of $f_{ij}$ in the $k$-direction, $i,j,k=1,2$. It is not difficult to verify that $f_{112}=f_{121}=f_{211}$ and similarly, $f_{122}=f_{212}=f_{221}$.

\subsection{Asymptotic nets}

A net $q:D\subset\Z^2\To\R^3$ is called asymptotic if the ``crosses are planar'', \emph{i.e.}, $q_1(u+\frac{1}{2},v)$, $q_1(u-\tfrac{1}{2},v)$, $q_2(u,v+\tfrac{1}{2})$ and $q_2(u,v-\tfrac{1}{2})$ are coplanar (see \cite{Bobenko2008}, p.66). From the coplanarity we obtain that
$$
\left[\,q_1\left(u\pm\tfrac{1}{2},v\right),\,q_2\left(u,v\pm\tfrac{1}{2}\right),\,q_{11}\left(u,v\right)\,\right]
=\left[\,q_1\left(u\pm\tfrac{1}{2},v\right),\,q_2\left(u,v\pm\tfrac{1}{2}\right),\,q_{22}\left(u,v\right)\,\right]=0,
$$
and we can assume
$$
M\left(u+\tfrac{1}{2},v+\tfrac{1}{2}\right)=\left[\,q_1\left(u+\tfrac{1}{2},v\right),\,q_2\left(u,v+\tfrac{1}{2}\right),\,q_{12}\left(u+\tfrac{1}{2},v+\tfrac{1}{2}\right)\,\right]>0.
$$

Similarly to the smooth case, the \emph{affine metric} $\Omega$ at a quadrangle $\left(u+\tfrac{1}{2},v+\tfrac{1}{2}\right)$ is defined by
\begin{equation}\label{affine-metric}
  \Omega\left(u+\tfrac{1}{2},v+\tfrac{1}{2}\right)=\sqrt{M\left(u+\tfrac{1}{2},v+\tfrac{1}{2}\right)}.
\end{equation}
We also define the affine normal vector by
\begin{equation}\label{normal-xie}
    \xi\left(u+\tfrac{1}{2},v+\tfrac{1}{2}\right)=\frac{q_{12}\left(u+\tfrac{1}{2},v+\tfrac{1}{2}\right)}{\Omega\left(u+\tfrac{1}{2},v+\tfrac{1}{2}\right)},\phantom{aaaa}
\end{equation}
(see \cite {Craizer2010}).

\subsection{Discrete improper affine spheres}

An asymptotic net is said to be a \emph{discrete improper affine sphere} (DIIAS) if the affine normal $\xi$ is constant.
From now on we shall be considering only this case, and we shall denote $\xi(u+\tfrac{1}{2},v+\tfrac{1}{2})$ by $\xi$. Thus we can write
$$
q_{12}\left(u+\tfrac{1}{2},v+\tfrac{1}{2}\right)=\Omega\left(u+\tfrac{1}{2},v+\tfrac{1}{2}\right)\xi.
$$
We define the coefficients of the cubic form by
$$
A(u,v)=\left[q_1\left(u-\tfrac{1}{2},v\right),q_1\left(u+\tfrac{1}{2},v\right),\xi\right],\ \ B(u,v)=\left[q_2\left(u,v-\tfrac{1}{2}\right),q_2\left(u,v+\tfrac{1}{2}\right),\xi\right].
$$

\begin{lem}
$A=A(u)$ and $B=B(v)$.
\end{lem}

\begin{proof}
We prove that $A_2(u,v+\tfrac{1}{2})=0$, the case for $B$ being similar. So $A_2(u,v+\tfrac{1}{2})=0$ is given by
$$
\left[q_1\left(u-\tfrac{1}{2},v+1\right),q_1\left(u+\tfrac{1}{2},v+1\right),\xi\right]-\left[q_1\left(u-\tfrac{1}{2},v+1\right),q_1\left(u+\tfrac{1}{2},v\right),\xi\right]+
$$
$$
+\left[q_1\left(u-\tfrac{1}{2},v+1\right),q_1\left(u+\tfrac{1}{2},v\right),\xi\right]-\left[q_1\left(u-\tfrac{1}{2},v\right),q_1\left(u+\tfrac{1}{2},v\right),\xi\right].
$$
Since $q_1\left(u\pm\tfrac{1}{2},v+1\right)-q_1\left(u\pm\tfrac{1}{2},v\right)=\Omega(u\pm\tfrac{1}{2},v+\tfrac{1}{2})\xi$, the lemma is proved.
\end{proof}

\begin{prop}\label{prop-eq-structural}
  The structural equations of the DIIAS are given by
\begin{equation}\label{eq-structural1}
\begin{array}{c}
  q_{11}(u,v)=\dfrac{\Omega_1\left(u,v+\tfrac{1}{2}\right)}{\Omega\left(u\pm\tfrac{1}{2},v+\tfrac{1}{2}\right)}\,q_1\left(u\pm\tfrac{1}{2},v\right)+
  \dfrac{A(u)}{\Omega\left(u\pm\tfrac{1}{2},v+\tfrac{1}{2}\right)}\,q_2\left(u,v+\tfrac{1}{2}\right) \\
 q_{11}(u,v)=\dfrac{\Omega_1\left(u,v-\tfrac{1}{2}\right)}{\Omega\left(u\pm\tfrac{1}{2},v-\tfrac{1}{2}\right)}\,q_1\left(u\pm\tfrac{1}{2},v\right)+
  \dfrac{A(u)}{\Omega\left(u\pm\tfrac{1}{2},v-\tfrac{1}{2}\right)}\,q_2\left(u,v-\tfrac{1}{2}\right)
\end{array}
\end{equation}
and
\begin{equation}\label{eq-structural2}
\begin{array}{c}
q_{22}(u,v)= \dfrac{B(v)}{\Omega\left(u+\tfrac{1}{2},v\pm\tfrac{1}{2}\right)}\,q_1\left(u+\tfrac{1}{2},v\right)+
  \dfrac{\Omega_2\left(u+\tfrac{1}{2},v\right)}{\Omega\left(u+\tfrac{1}{2},v\pm\tfrac{1}{2}\right)}\,q_2\left(u,v\pm\tfrac{1}{2}\right)\\
  q_{22}(u,v)= \dfrac{B(v)}{\Omega\left(u-\tfrac{1}{2},v\pm\tfrac{1}{2}\right)}\,q_1\left(u-\tfrac{1}{2},v\right)+
  \dfrac{\Omega_2\left(u+\tfrac{1}{2},v\right)}{\Omega\left(u-\tfrac{1}{2},v\pm\tfrac{1}{2}\right)}\,q_2\left(u,v\pm\tfrac{1}{2}\right)
\end{array}
\end{equation}
\end{prop}

For a proof of this proposition, see \cite{Craizer2010}. There is also a compatibility equation analogous to the Equation \ref{Compatibility-eq-smooth} in the smooth case:

\begin{prop}
The following compatibility equation holds:
$$
\Omega\left(u+\tfrac{1}{2},v-\tfrac{1}{2}\right)\Omega\left(u-\tfrac{1}{2},v+\tfrac{1}{2}\right)-\Omega\left(u+\tfrac{1}{2},v+\tfrac{1}{2}\right)\Omega\left(u-\tfrac{1}{2},v-\tfrac{1}{2}\right)=A(u)B(v).
$$
\end{prop}

\begin{proof}
We know that
$$
q_{121}(u,v+\tfrac{1}{2})=\Omega_1(u,v+\tfrac{1}{2})\xi.
$$
On the other hand, Equations \eqref{eq-structural1} imply that $q_{112}(u,v+\tfrac{1}{2})$ is given by
$$
\dfrac{\Omega_1\left(u,v+\tfrac{1}{2}\right)}{\Omega\left(u+\tfrac{1}{2},v+\tfrac{1}{2}\right)}\,q_1\left(u+\tfrac{1}{2},v+1\right)-\dfrac{\Omega_1\left(u,v-\tfrac{1}{2}\right)}{\Omega\left(u+\tfrac{1}{2},v-\tfrac{1}{2}\right)}\,q_1\left(u+\tfrac{1}{2},v\right)+
$$
$$
+A(u) \left(  \dfrac{q_2\left(u,v+\tfrac{1}{2}\right)}{\Omega\left(u+\tfrac{1}{2},v+\tfrac{1}{2}\right)}- \dfrac{q_2\left(u,v-\tfrac{1}{2}\right)}{\Omega\left(u+\tfrac{1}{2},v-\tfrac{1}{2}\right)}\, \right).
$$
These expression can be written as
$$
\frac{\Omega_1(u,v+\tfrac{1}{2})}{\Omega(u+\tfrac{1}{2},v+\tfrac{1}{2})}q_{12}(u+\tfrac{1}{2},v+\tfrac{1}{2})+\left(  \frac{\Omega_1(u,v+\tfrac{1}{2})}{\Omega(u+\tfrac{1}{2},v+\tfrac{1}{2})}-\frac{\Omega_1(u,v-\tfrac{1}{2})}{\Omega(u+\tfrac{1}{2},v-\tfrac{1}{2})}  \right)q_1(u+\tfrac{1}{2},v)
$$
$$
+\frac{A(u)}{\Omega(u+\tfrac{1}{2},v+\tfrac{1}{2})}q_{22}(u,v)-\frac{A(u)\Omega_2(u+\tfrac{1}{2},v)}{\Omega(u+\tfrac{1}{2},v+\tfrac{1}{2})\Omega(u+\tfrac{1}{2},v-\tfrac{1}{2})}q_2(u,v-\tfrac{1}{2}).
$$
Now take the component $q_1(u+\tfrac{1}{2},v)$ in this expression to obtain
$$
\frac{\Omega_1(u,v+\tfrac{1}{2})}{\Omega(u+\tfrac{1}{2},v+\tfrac{1}{2})}-\frac{\Omega_1(u,v-\tfrac{1}{2})}{\Omega(u+\tfrac{1}{2},v-\tfrac{1}{2})}+\frac{A(u)B(v)}{\Omega(u+\tfrac{1}{2},v+\tfrac{1}{2})\Omega(u+\tfrac{1}{2},v-\tfrac{1}{2})}=0,
$$
which proves the proposition.
\end{proof}

\subsection{The $x$-net and the $q$-net}

Consider a DIIAS with $\xi=(0,0,1)$ and write $q(u,v)=(x(u,v),z(u,v)$, with $x(u,v)\in\R^2$, $z(u,v)\in\R$.
The planar net $x(u,v)$, $(u,v)\in D\subset\mathbb{Z}^2$, will be called the $x$-net. The asymptotic spatial net $q(u,v)$, $(u,v)\in D\subset\mathbb{Z}^2$, will be called the $q$-net. Since $x_{12}(u,v)=0$, the quadrangles of the $x$-net are all paralellograms.

\subsection{Bilinear patches}

A quadrangle is defined by its vertices $(u,v)$, $(u+1,v)$, $(u,v+1)$ and $(u+1,v+1)$ in the domain $D$. For short, we shall sometimes denote it by its center $(u+\tfrac{1}{2},v+\tfrac{1}{2})$.
For each such quadrangle,
there exists a unique bilinear patch contained in a hyperbolic paraboloid with affine normal $\xi$ and passing through $q(u,v)$, $q(u+1,v)$, $q(u,v+1)$ and $q(u+1,v+1)$. We shall denote this bilinear patch by $BP=BP(u+\tfrac{1}{2},v+\tfrac{1}{2})$. A parameterization of this bilinear patch is given by
$$
BP(s,t)=q(u,v)+sq_1\left(u+\tfrac{1}{2},v\right)+tq_2\left(u,v+\tfrac{1}{2}\right)+st q_{12}\left(u+\tfrac{1}{2},v+\tfrac{1}{2}  \right),
$$
$0\leq s,t\leq 1$.
The tangent plane to $BP$ at $(u,v)$ contains $q_1(u+\tfrac{1}{2},v)$ and $q_2(u,v+\tfrac{1}{2})$, thus coinciding with the star plane at $(u,v)$. At a point of the edge $(u+\tfrac{1}{2},v)$, both $BP(u+\tfrac{1}{2},v+\tfrac{1}{2})$ and $BP(u+\tfrac{1}{2},v-\tfrac{1}{2})$ share the same tangent plane, ``linear interpolators'' of the star planes at $(u,v)$ and $(u+1,v)$. We conclude that the bilinear patches are glued at the edges with the same tangent planes (\cite{Craizer2010}, \cite{Rorig2014}, \cite{Kaferbock2013}).

In this article, the bilinear patches are used just to visualize the discrete cuspidal edges and discrete swallowtails in Section 4.

\subsection{Discrete centre-chord construction}

We describe now the centre-chord construction for a general DIIAS, which can also be found in \cite{Kobayashi2020}.
Let $\alpha:I\To\R^2$ and $\beta:J\To\R^2$, where $I,J\subset\Z$, and define
\begin{equation}\label{eq:definexy}
x(u,v)=\frac{1}{2}(\alpha(u)+\beta(v)),\ \ \ y(u,v)=\frac{1}{2}(\beta(v)-\alpha(u)).
\end{equation}
Define a function $z:I\times J\To\R$ by the conditions
\begin{equation}\label{eq:definez}
z_1(u+\tfrac{1}{2},v)=[x_1(u+\tfrac{1}{2},v),y(u,v)],\ \ z_2(u,v+\tfrac{1}{2})=[x_2(u,v+\tfrac{1}{2}),y(u,v)].
\end{equation}
To prove the existence of such function $z$, we must show that $z_{12}=z_{21}$. But
$$
z_{12}(u+\tfrac{1}{2},v+\tfrac{1}{2})=[x_1(u+\tfrac{1}{2},v+1),y(u,v+1)]-[x_1(u+\tfrac{1}{2},v),y(u,v)]
$$
$$
=\frac{1}{4}[\alpha_1(u+\tfrac{1}{2}),\beta_2(v+\tfrac{1}{2})],
$$
and similarly
$$
z_{21}(u+\tfrac{1}{2},v+\tfrac{1}{2})=\frac{1}{4}[\alpha_1(u+\tfrac{1}{2}),\beta_2(v+\tfrac{1}{2})],
$$
which prove that $z$ is well-defined.

Define
\begin{equation}\label{eq:MetricCenterChord}
\Omega(u+\tfrac{1}{2},v+\tfrac{1}{2})=\tfrac{1}{4}\ [\alpha_1(u+\tfrac{1}{2}),\beta_2(v+\tfrac{1}{2})].
\end{equation}
\begin{prop} \label{centre-chord-DIIIAS}
Assume $\Omega(u+\tfrac{1}{2},v+\tfrac{1}{2})>0$, for any $(u,v)\in I\times J$. Then the net $q(u,v)=(x(u,v),z(u,v))$ defines a DIIAS with cubic form
$$
A=\tfrac{1}{4}[\alpha_1(u-\tfrac{1}{2}),\alpha_1(u+\tfrac{1}{2})], \ \ B=-\tfrac{1}{4}[\beta_2(v-\tfrac{1}{2}),\beta_2(v+\tfrac{1}{2})].
$$
\end{prop}

\begin{proof}
Writing
$$
\alpha_{11}(u,v)=\lambda_{\pm}^1\alpha_1(u\pm\tfrac{1}{2})+\mu_{\pm}^1\beta_2(v\pm\tfrac{1}{2}),
$$
we obtain
$$
\lambda_{\pm}^1=\frac{\Omega_1(u,v\pm\tfrac{1}{2})}{\Omega(u\pm\tfrac{1}{2},v\pm\tfrac{1}{2})},\ \
\mu_{\pm}^1=\frac{A(u,v)}{\Omega(u\pm\tfrac{1}{2},v\pm\tfrac{1}{2})}.
$$
Similarly, we have that
$$
\beta_{22}(u,v)=\lambda_{\pm}^2\alpha_1(u\pm\tfrac{1}{2})+\mu_{\pm}^2\beta_2(v\pm\tfrac{1}{2}),
$$
where
$$
\lambda_{\pm}^2=\frac{B(u,v)}{\Omega(u\pm\tfrac{1}{2},v\pm\tfrac{1}{2})},\ \
\mu_{\pm}^2=\frac{\Omega_2(u\pm\tfrac{1}{2},v)}{\Omega(u\pm\tfrac{1}{2},v\pm\tfrac{1}{2})}.
$$
Using Equations \eqref{eq:definexy} and \eqref{eq:definez}, the above equations imply that
$$
q_{11}(u,v)=\lambda_{\pm}^1q_1(u\pm\tfrac{1}{2})+\mu_{\pm}^1q_2(v\pm\tfrac{1}{2}), \ \ q_{22}(u,v)=\lambda_{\pm}^2q_1(u\pm\tfrac{1}{2})+\mu_{\pm}^2q_2(v\pm\tfrac{1}{2}),
$$
thus proving that $q$ is an asymptotic net. Moreover, since $x_{12}=0$ and $z_{12}=\Omega$, we obtain
$$
q_{12}(u+\tfrac{1}{2},v+\tfrac{1}{2})=\Omega(u+\tfrac{1}{2},v+\tfrac{1}{2})\xi,
$$
where $\xi=(0,0,1)$, thus proving that $q$ is a DIIAS. Finally, from the expressions of $\mu_{\pm}^1$ and $\lambda_{\pm}^2$ we conclude that the cubic form is given by $A$ and $B$.
\end{proof}

The converse of Proposition \ref{centre-chord-DIIIAS} also holds:

\begin{prop}\label{DIIIAS-centre-chord}
 Any DIIAS $q:I\times J\To\R^3$ can be obtained by the centre-chord construction, that is, $q(u,v)=(x(u,v),z(u,v))$,
  where $x(u,v)=\frac{1}{2}(\alpha(u)+\beta(v))$, $z_1(u+\tfrac{1}{2},v)=[x_1(u+\tfrac{1}{2},v),y(u,v)]$,  $z_2(u,v+\tfrac{1}{2})=[x_2(u,v+\tfrac{1}{2}),y(u,v)]$ and $y(u,v)=\frac{1}{2}(\beta(v)-\alpha(u))$, for some polygonal lines $\alpha:I\To\R^2$ and $\beta:J\To\R^2$, where $I,J\subset\Z$.
\end{prop}

\noindent\emph{Proof.}
We may assume that $\xi=(0,0,1)$ and so $q(u,v)=(x(u,v),z(u,v))$, where $x(u,v)$ is the projection in the plane $z=0$.
Since
$$
q_{12}(u+\tfrac{1}{2},v+\tfrac{1}{2})=\Omega(u+\tfrac{1}{2},v+\tfrac{1}{2})\,\xi,$$
we conclude that $x_{12}=0$, which implies that
$$x(u,v)=\frac{1}{2}(\alpha(u)+\beta(v)),$$
for some polygonal lines $\alpha:I\To\R^2$ and $\beta:J\To\R^2$, where $I,J\subset\Z$.
Since
$$\begin{array}{rcl}
    q_1(u+\tfrac{1}{2},v) &=& \frac{1}{2}\left(\alpha_1(u+\tfrac{1}{2}),\,z_1(u+\tfrac{1}{2})\right),\\
    q_2(u,v+\tfrac{1}{2}) &=& \frac{1}{2}\left(\beta_2(v+\tfrac{1}{2}),\,z_2(v+\tfrac{1}{2})\right), \phantom{\dfrac{1}{1}}\\
    q_{12}(u+\tfrac{1}{2},v+\tfrac{1}{2}) &=& \left(0,\,\Omega(u+\tfrac{1}{2},v+\tfrac{1}{2})\right),
  \end{array}$$
we conclude that
$$\Omega^2(u+\tfrac{1}{2},v+\tfrac{1}{2})=[q_1,q_2,q_{12}]=\tfrac{1}{4}\Omega(u+\tfrac{1}{2},v+\tfrac{1}{2})\,\left[\alpha_1(u+\tfrac{1}{2}),\,\beta_2(v+\tfrac{1}{2})\right]$$
and so
$$\Omega(u+\tfrac{1}{2},v+\tfrac{1}{2})=\tfrac{1}{4}\left[\alpha_1(u+\tfrac{1}{2}),\,\beta_2(v+\tfrac{1}{2})\right]>0.$$
Thereafter,
$$z_{12}(u+\tfrac{1}{2},v+\tfrac{1}{2}))=\tfrac{1}{4}\left[\alpha_1(u+\tfrac{1}{2}),\,\beta_2(v+\tfrac{1}{2})\right],$$
and by discrete integration we obtain
$$
 z_1(u+\tfrac{1}{2},v)=\left[x_1(u+\tfrac{1}{2}),\,y(u,v)\right],\ \  z_2(u,v+\tfrac{1}{2})=\left[x_2(v+\tfrac{1}{2}),\,y(u,v)\right],
$$
where $y(u,v)=\frac{1}{2}(\beta(v)-\alpha(u))$, which completes the proof.\hfill$\square$

\section{Singularities of discrete improper indefinite affine spheres}

Even if we do not assume the hypothesis $\Omega>0$, we shall call the asymptotic net obtained by the centre-chord construction a DIIAS. In this context, ``singularities'' may appear.

Consider two discrete planar polygonal lines $\alpha:I\To\R^2$ and $\beta:J\To\R^2$, where $I,J\subset\Z$. We shall assume that:
\begin{itemize}
\item[$\bullet$] For any point $\alpha(u)$ and any triplet $\beta(v-1)$, $\beta(v)$, $\beta(v+1)$, terefore $\beta(v)$ is within the interior of the angle $\beta(v-1)\alpha(u)\beta(v+1)$, supposedly less than $180^\circ$.
\item[$\bullet$] For any point $\beta(v)$ and any triplet $\alpha(u-1)$, $\alpha(u)$, $\alpha(u+1)$, therefore $\alpha(u)$ is within the interior of the angle $\alpha(u-1)\beta(v)\alpha(u+1)$, supposedly less than $180^\circ$. See Figure \ref{Fig-condition-planar-curves}.
\end{itemize}

\begin{figure}[!htb]
 \centering
 \includegraphics[width=.5\linewidth]{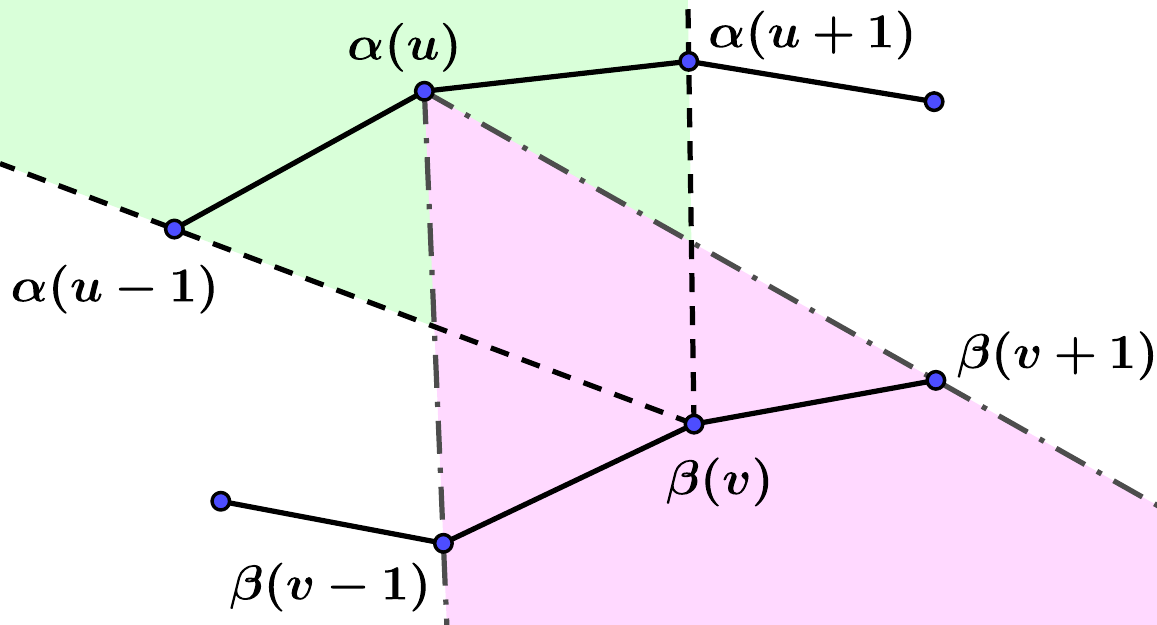}
\caption{\small Restriction to the pair of planar curves $(\alpha,\beta)$.}\label{Fig-condition-planar-curves}
\end{figure}

This restriction is made to simplify our first model of singularities, but we think that it is something to be explored in future works about the subject.

\subsection{Singular (cuspidal) edges of the asymptotic net}

An edge in the domain $D$ with endpoints $(u,v)$ and $(u+1,v)$ will be written simply as $(u+\tfrac{1}{2},v)$. Similarly, an edge with endpoints $(u,v)$ and $(u,v+1)$ will be written as $(u,v+\tfrac{1}{2})$.

The singular set of a smooth IIAS is defined by the condition $\Omega(u,v)=0$. The corresponding discrete definition is the following (see \cite{Rossman2018}):

\begin{defn}
An edge $(u+\tfrac{1}{2},v)$ is called {\it singular} if
$$
\Omega(u+\tfrac{1}{2},v+\tfrac{1}{2})\cdot\Omega(u+\tfrac{1}{2},v-\tfrac{1}{2})<0.
$$
Similarly, an edge $(u,v+\tfrac{1}{2})$ is called {\it singular} if
$$
\Omega(u+\tfrac{1}{2},v+\tfrac{1}{2})\cdot\Omega(u-\tfrac{1}{2},v+\tfrac{1}{2})<0.
$$
\end{defn}

\begin{rem} In this paper, we shall consider only asymptotic nets with $\Omega(u+\tfrac{1}{2},v+\tfrac{1}{2})\neq 0$, for any quadrangle $(u+\tfrac{1}{2},v+\tfrac{1}{2})$ in the domain $D$. Observe that this condition holds for generic asymptotic nets.
\end{rem}

In the smooth case the condition $\Omega(u,v)=0$ is equivalent to $\tfrac{d\alpha}{du}$ and $\tfrac{d\beta}{dv}$ being parallel. In the discrete setting we have the following lemma, whose proof is immediate from Equation \eqref{eq:MetricCenterChord}:

\begin{lem}\label{LemmaHalfParallel}
Let $\alpha:I\to\mathbb{R}^2$ and $\beta:J\to\mathbb{R}^2$, $I,J\subset\mathbb{Z}$, be polygonal lines. Then:

\begin{enumerate}
\item\label{HalfParallel1}
An edge $(u+\tfrac{1}{2},v)$ is singular if and only if $\beta(v-1)$ and $\beta(v+1)$ are in the same half-plane determined by the straight line given by $\beta(v)+r\alpha_1(u+\tfrac{1}{2})$, for $r\in\R$.

\item\label{HalfParallel2}
An edge $(u,v+\tfrac{1}{2})$ is singular if and only if if $\alpha(u-1)$ and $\alpha(u+1)$ are in the same half-plane determined by the straight line given by $\alpha(u)+r\beta_2(v+\tfrac{1}{2})$, for $r\in\R$.

\end{enumerate}

\end{lem}

When condition (\ref{HalfParallel1}) of Lemma \ref{LemmaHalfParallel} holds, we say that
$\alpha_1(u+\tfrac{1}{2})$ is {\it discretely parallel} to $\beta(v)$. When condition (\ref{HalfParallel2}) of Lemma \ref{LemmaHalfParallel} holds,
we say that $\beta_2(v+\tfrac{1}{2})$ is {\it discretely parallel} to $\alpha(u)$ (see Figure \ref{Fig-sing-dmptl}).

\begin{figure}[!htb]
\centering
 \includegraphics[width=.45\linewidth]{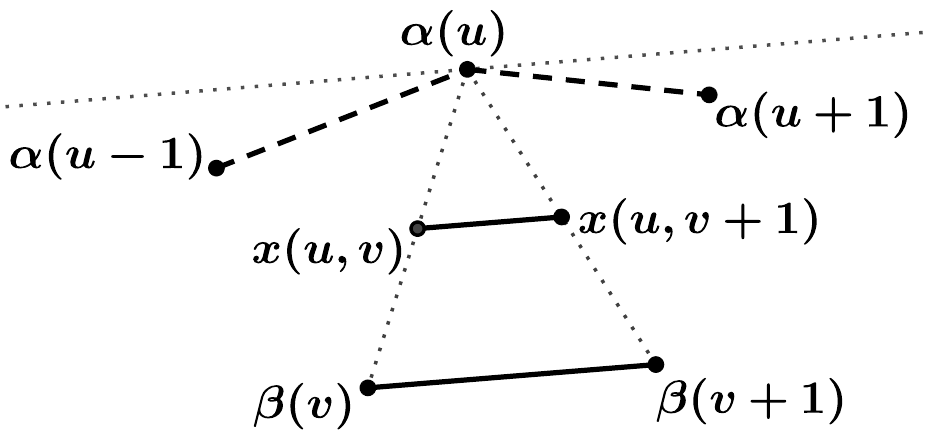}
\caption{\small Discrete parallelism between $\beta_2(v+\tfrac{1}{2})$ and $\alpha(u)$.  The edge $x_2(u,v+\tfrac{1}{2})$ belongs to the DMPTL of the pair of polygonal lines $(\alpha,\beta)$. The line styles indicate the $u$ and $v$ directions.}\label{Fig-sing-dmptl}
\end{figure}

Observe that the discrete parallelism is associated to a triangle formed by the point of one polygonal line and an edge of the other, as we can see in Figure \ref{Fig-sing-dmptl}.

We shall call the set of all midsegments of these triangles the \emph{discrete midpoint parallel tangent locus} (DMPTL) of the pair of polygonal lines $(\alpha,\beta)$.

We can characterize an edge of the DMPTL among all edges of the $x$-net by the following proposition:

\begin{prop}
Consider an edge $x_2(u,v+\tfrac{1}{2})$ of the $x$-net and denote the straight line containing it by $r(u,v+\tfrac{1}{2})$. The following statements are equivalent:
\begin{enumerate}
\item
The edge $x_2(u,v+\tfrac{1}{2})$  is an edge of the DMPTL.

\item
The straight line $r(u,v+\tfrac{1}{2})$ leaves $x(u-1,v)$ and $x(u+1,v)$ at the same half-plane.

\item
The straight line $r(u,v+\tfrac{1}{2})$ leaves $x(u-1,v+1)$ and $x(u+1,v+1)$ at the same half-plane.
\end{enumerate}
\end{prop}

\begin{proof}
The proof follows directly from the fact that the quadrangles of the $x$-net are parallelograms with $x_1(u+\frac{1}{2},v)$ parallel to $\alpha_1(u+\tfrac{1}{2})$.
\end{proof}

\begin{cor}\label{HalfStarPlane}
Consider an edge $q_2(u,v+\tfrac{1}{2})$ of the asymptotic net and denote the straight line containing it by $s(u,v+\tfrac{1}{2})$. The following statements are equivalent:
\begin{enumerate}
\item
The segment $x_2(u,v+\tfrac{1}{2})$  is an edge of the DMPTL.

\item
In the star plane at $q(u,v)$, the straight line $s(u,v+\tfrac{1}{2})$ leaves $q(u-1,v)$ and $q(u+1,v)$ at the same half-plane.

\item
In the star plane at $q(u,v+1)$, the straight line $s(u,v+\tfrac{1}{2})$ leaves $q(u-1,v+1)$ and $q(u+1,v+1)$ at the same half-plane.
\end{enumerate}
\end{cor}

We say that an edge $q_2(u,v+\tfrac{1}{2})$ of the DIIAS is \emph{singular} if it satisfies one, and hence all, the conditions of Corollary \ref{HalfStarPlane}. With this definition, it is clear that an edge of the DIIAS is singular if and only if its projection is an edge of the DMPTL, which is a discrete counterpart of the correponding property of the smooth case.

Observe that we can check whether or not the edge $q_2(u,v+\tfrac{1}{2})$ is singular by looking at the star plane at $q(u,v)$ or at the star plane at $q(u,v+1)$. Thus the above definition (items (2) and (3) of Corollary \ref{HalfStarPlane}) of a singular edge can be directly extended to any asymptotic net, even if it does not represent a DIIAS.

The singular edges of a DIIAS are the discrete counterpart of the cuspidal edges of the smooth IIAS. So, in the discrete setting, the expressions singular edge and {\it cuspidal edge} have the same meaning.

\begin{figure}[!htb]
\centering
 \includegraphics[width=.50\linewidth]{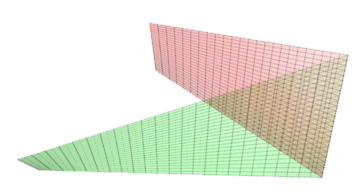}
\caption{\small Discrete cuspidal edge with the corresponding two bilinear patches, distinguished by the colors.}\label{Fig:DiscreteSwallowtailInterpolator0}
\end{figure}

In Figure \ref{Fig:DiscreteSwallowtailInterpolator0}, one can see how a discrete cuspidal edge with the bilinear interpolators looks like a smooth cuspidal edge. It is easy to verify that the two bilinear patches associated with a cuspidal edge are in the same side of the vertical plane containing this edge, and in fact this property characterizes cuspidal edges.

\subsection{Singular polygonal line}

Let us construct the DMPTL step by step. Suppose that $\alpha_1(u-\tfrac{1}{2})$ is discretely parallel to $\beta(v)$ for some $u$ and $v$, then we have formed a triangle and its midsegment is part of the DMPTL. The next step is to decide which adjacent triangle we should choose and that is going to be made clear after the following Proposition.

\begin{prop}\label{Prop-dmptl-construction}
Let $\alpha$ and $\beta$ be planar polygonal lines such that $\alpha_1(u-\tfrac{1}{2})$ is discretely parallel to $\beta(v)$. Then one and only one of the three follow statements holds:
\begin{enumerate}
\item $\alpha_1(u+\tfrac{1}{2})$ is discretely parallel to $\beta(v)$, as in Figure \ref{Fig-dmptl-first}(left);
\item $\beta_2(v+\tfrac{1}{2})$ is discretely parallel to $\alpha(u)$, as in Figure \ref{Fig-dmptl-first}(centre);
\item $\beta_2(v-\tfrac{1}{2})$ is discretely parallel to $\alpha(u)$, as in Figure \ref{Fig-dmptl-first}(right).
\end{enumerate}
\end{prop}

\noindent\emph{Proof.}
Let us fix $\alpha(u-1)$, $\alpha(u)$, $\beta(v-1)$, $\beta(v)$ and $\beta(v+1)$ so that the hypothesis continues valid and then observe what can happen to the point $\alpha(u+1)$ (see Figure \ref{Fig-sing-proposition-c-possib}).

\begin{figure}[!htb]
\centering
\includegraphics[width=.4\linewidth]{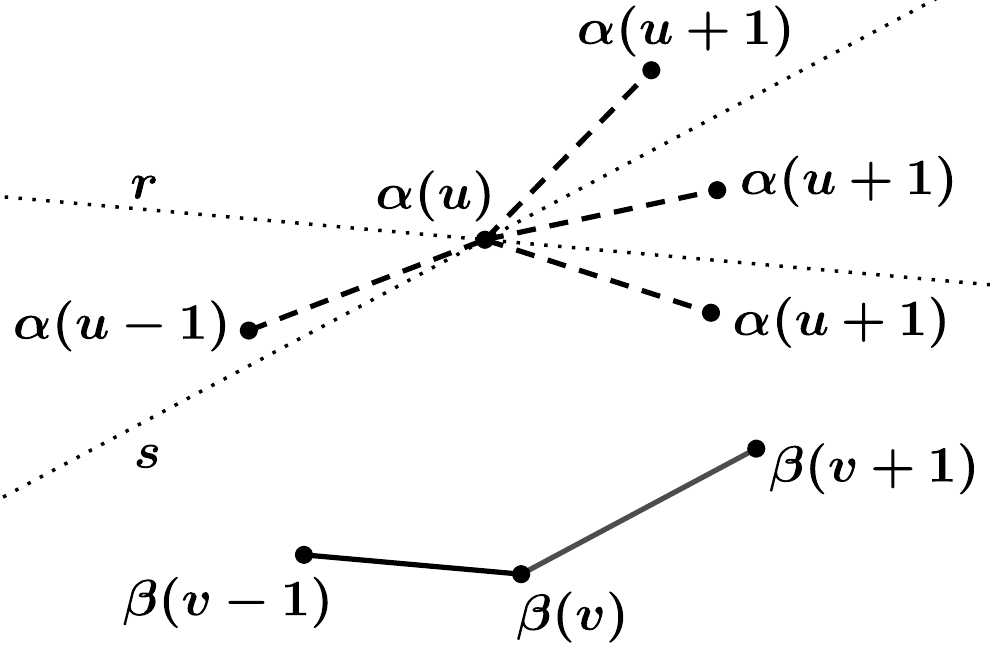}
\caption{\small The three possible configurations for the vertex $\alpha(u+1)$.}\label{Fig-sing-proposition-c-possib}
\end{figure}

Let $r$ and $s$ be straight lines passing through $\alpha(u)$ and parallel to the edges $\beta_2(v-\tfrac{1}{2})$ and $\beta_2(v+\tfrac{1}{2})$, respectively. They divide the plane in four open regions and the point $\alpha(u+1)$ can be only in three of them. In fact, by the restriction made to the curves $\alpha$ and $\beta$ at the beginning of the section, it can not be in the same region wherein the point $\alpha(u-1)$ is. Also, since there is no parallelism between edges of $\alpha$ and $\beta$, $\alpha(u+1)$ can neither be in the straight line $r$ nor in $s$.

If $\alpha(u+1)$ belongs to the region opposed to the region where $\alpha(u-1)$ belongs, then $\alpha_1(u+\tfrac{1}{2})$ is discretely parallel to $\beta(v)$. Thus $x_1(u+\tfrac{1}{2},v)$ belongs to the DMPTL, while $x_2(u,v\pm\tfrac{1}{2})$ does not and this takes us to case (1).

\begin{figure}[!htb]
\includegraphics[width=.32\linewidth]{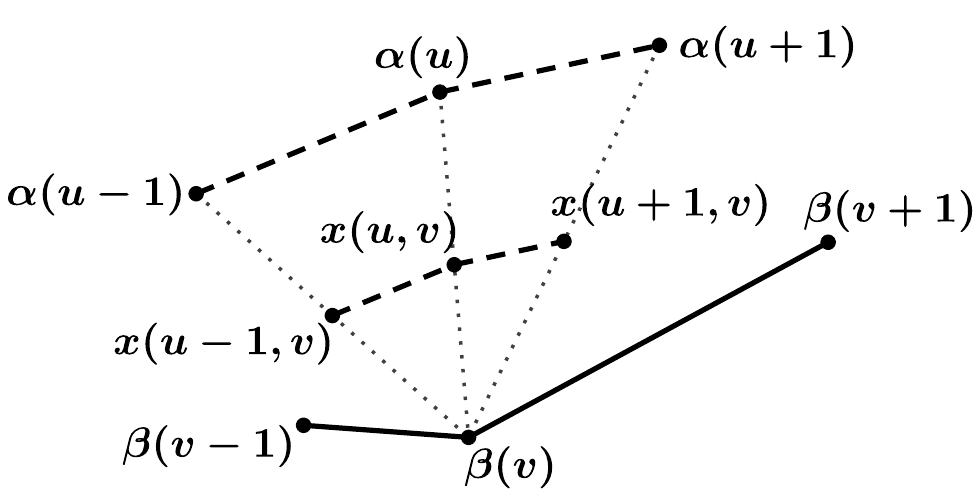}
\hfill
\includegraphics[width=.32\linewidth]{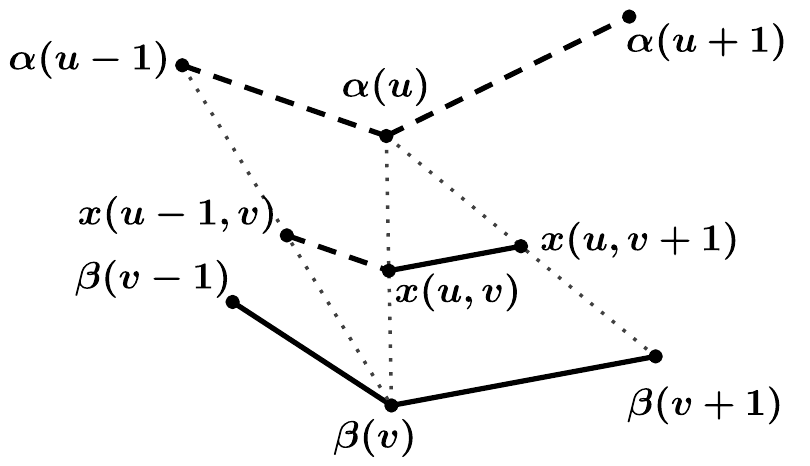}
\hfill
\includegraphics[width=.32\linewidth]{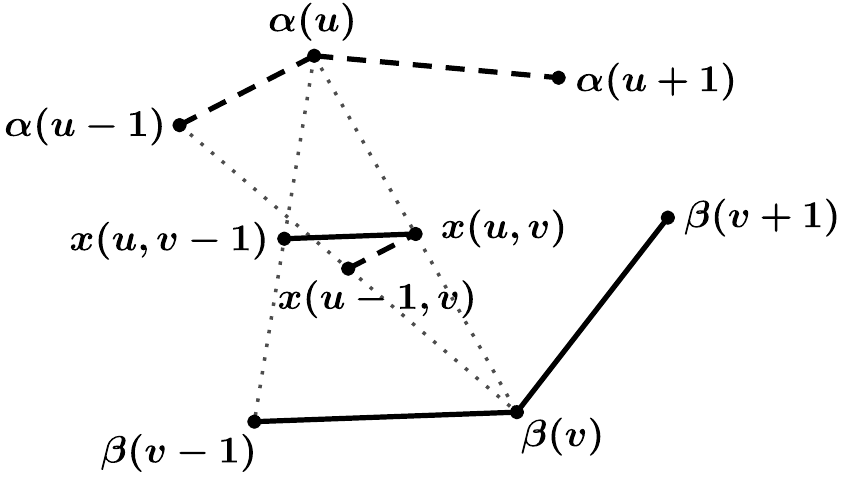}
\caption{\small Three possibilities for the construction of the DMPTL.}\label{Fig-dmptl-first}
\end{figure}

If $\alpha(u+1)$ is above $r$ and $s$, then $\beta_2(v+\tfrac{1}{2})$ is discretely parallel to $\alpha(u)$. Thus $x_2(u,v+\tfrac{1}{2})$ belongs to the DMPTL, while $x_1(u+\tfrac{1}{2},v)$ and $x_2(u,v-\tfrac{1}{2})$ does not, and this takes us to case (2).

If $\alpha(u+1)$ is below $r$ and $s$, then $\beta_2(v-\tfrac{1}{2})$ is discretely parallel to $\alpha(u)$. Thus $x_2(u,v-\tfrac{1}{2})$ belongs to the DMPTL, while $x_1(u+\tfrac{1}{2},v)$ and $x_2(u,v+\tfrac{1}{2})$ does not, and this takes us to case (3). \hfill$\square$

\begin{cor}
The DMPTL is a planar polygonal line. As a consequence, the set of singular edges of a DIIAS form a spatial polygonal line.
\end{cor}

\subsection{Configuration of a star}

Let us make some notes about possible configurations for star planes in the asymptotic net or in the planar net.
A star plane at $q(u,v)$ is called \emph{typical} if the four points $q(u+1,v)$, $q(u,v+1)$, $q(u-1,v)$ and $q(u,v-1)$ appear in this order, clockwise or counter clockwise, with respect to $q(u,v)$, and \emph{atypical} otherwise (see Figure \ref{Fig-typical-at-star}).

We can consider similarly typical and atypical vertices in the planar net. It is clear that $q(u,v)$ is typical for the asymptotic net if and only if $x(u,v)$ is typical for the planar net.

\begin{figure}[!htb]
\begin{minipage}[b]{0.24\linewidth}
\centering
\includegraphics[width=1\linewidth]{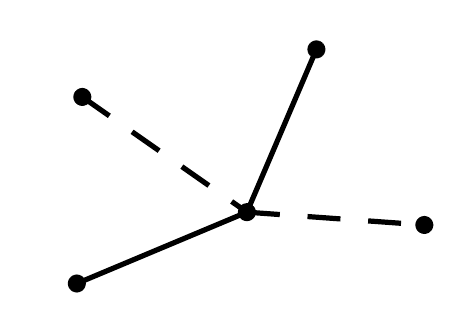}
\end{minipage}\hfill
\begin{minipage}[b]{0.23\linewidth}
\centering
\includegraphics[width=1\linewidth]{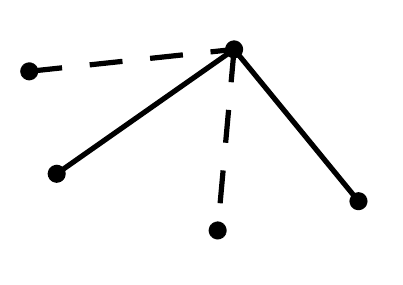}
\end{minipage}\hfill
\begin{minipage}[b]{0.24\linewidth}
\centering
\includegraphics[width=1\linewidth]{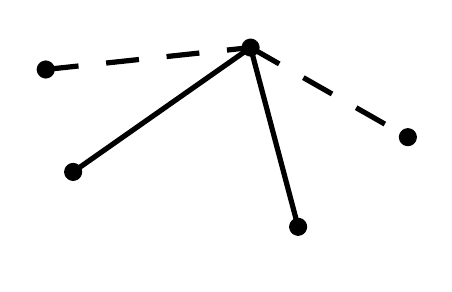}
\end{minipage}\hfill
\begin{minipage}[b]{0.23\linewidth}
\centering
\includegraphics[width=1\linewidth]{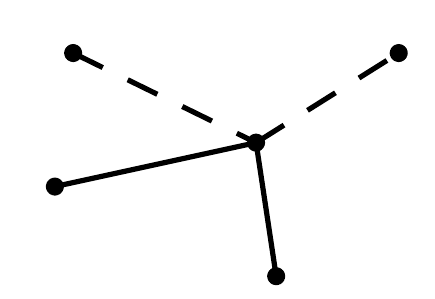}
\end{minipage}
\caption{\small Both figures on the left show two different possibilities of a typical star, whilst both on the right show two possible configurations for an atypical one. The line styles indicate the $u$ and $v$ directions.}\label{Fig-typical-at-star}
\end{figure}

We have the following proposition:
\begin{prop}
Consider a vertex $x(u,v)$ of the planar net. Then one and only one of the following conditions holds:
\begin{enumerate}
\item[$(0)$]
No edge in the star is in the DMPTL.
\item[$(1)$]
Two consecutive edges with the same label are in the DMPTL.
\item[$(2)$]
Two adjacent edges with different labels are in the DMPTL and the star is typical.
\item[$(3)$]
Two adjacent edges with different labels are in the DMPTL and the star is atypical.
\end{enumerate}
\end{prop}

\begin{proof}
If no edges of the star at $x(u,v)$ is in the DMPTL, we are in case (0). If at least one is in the DMPTL, we may assume it is $x_1(u-\tfrac{1}{2},v)$. Then proceeding as in the proof of Proposition \ref{Prop-dmptl-construction}, there are three possibilities, cases (1), (2) or (3).
In case (1), two consecutive edges with the same label are in the DMPTL. In case (2), two adjacent edges with different labels are in the DMPTL and the star is typical. Finally in case (3), two adjacent edges with different labels are in the DMPTL and the star is atypical.
\end{proof}

\begin{cor}\label{Configs}
Consider a vertex $q(u,v)$ of the asymptotic net. Then one and only one of the following conditions holds:
\begin{enumerate}
\item[$(0)$]
No edge in the star is singular.
\item[$(1)$]
Two consecutive edges with the same label are singular.
\item[$(2)$]
Two adjacent edges with different labels are singular and the star is typical.
\item[$(3)$]
Two adjacent edges with different labels are singular and the star is atypical.
\end{enumerate}
\end{cor}

Figures \ref{Fig-cuspidaledge-config-uu} and \ref{Fig:DiscreteCuspidalEdgesInterpolator} show a neighborhood of $q(u,v)$ satisfying conditions (1) and (2) of Corollary \ref{Configs}.

\begin{figure}[!htb]
 \includegraphics[width=.4\linewidth]{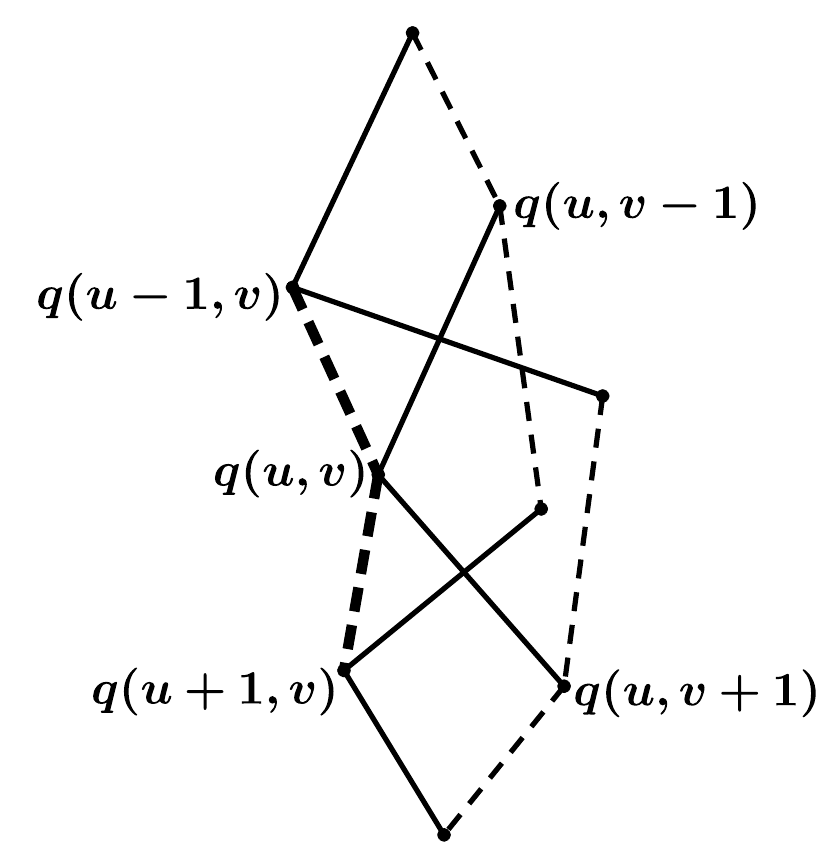}
 \hspace{1cm}
  \includegraphics[width=.4\linewidth]{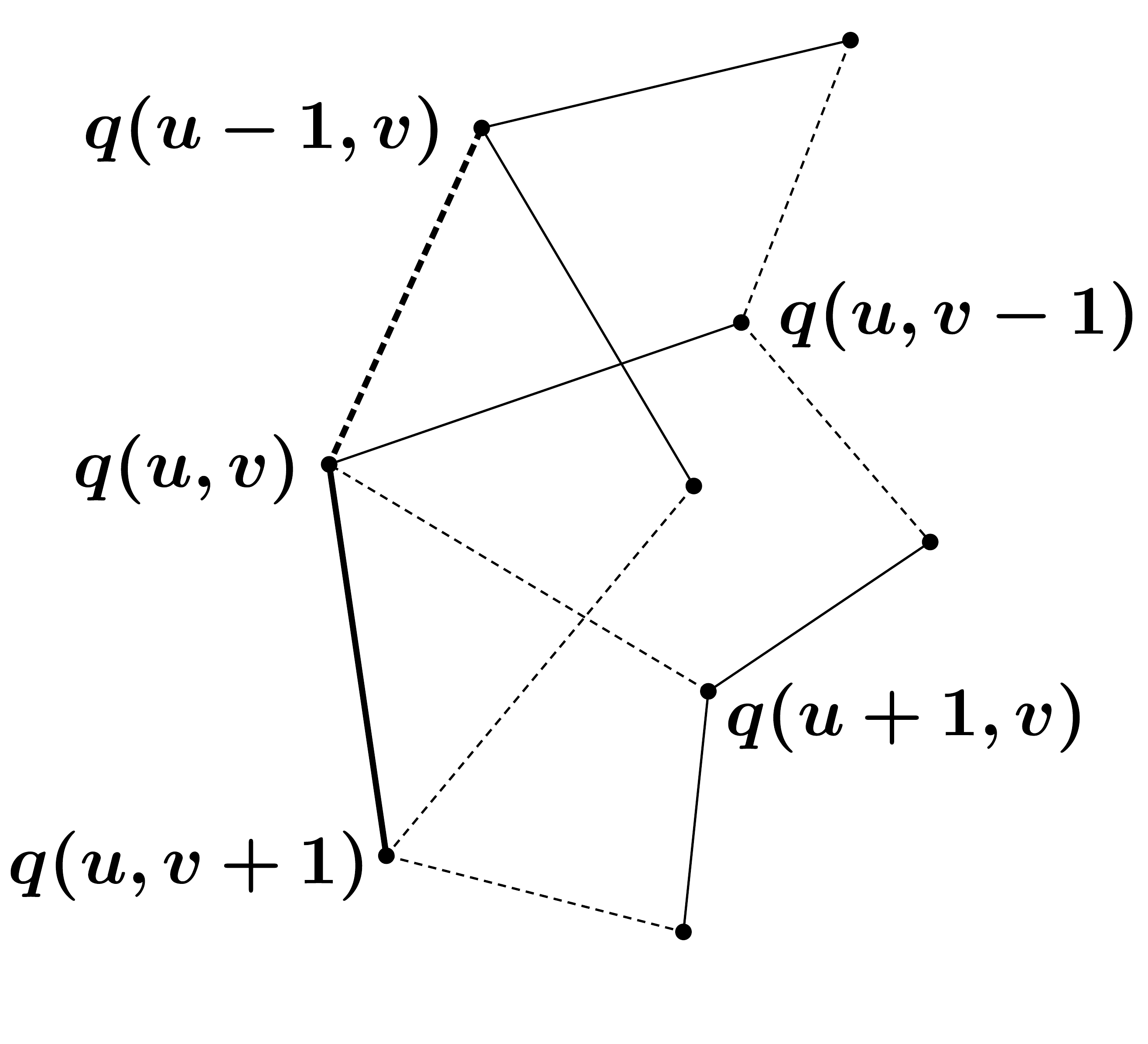}
\caption{\small A pair of cuspidal edges satisfying conditions (1) and (2) of Corollary \ref{Configs}, respectively.}\label{Fig-cuspidaledge-config-uu}
\end{figure}

\begin{figure}[!htb]
 \includegraphics[width=.28\linewidth]{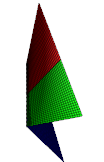}
 \hspace{2cm}
 \includegraphics[width=.4\linewidth]{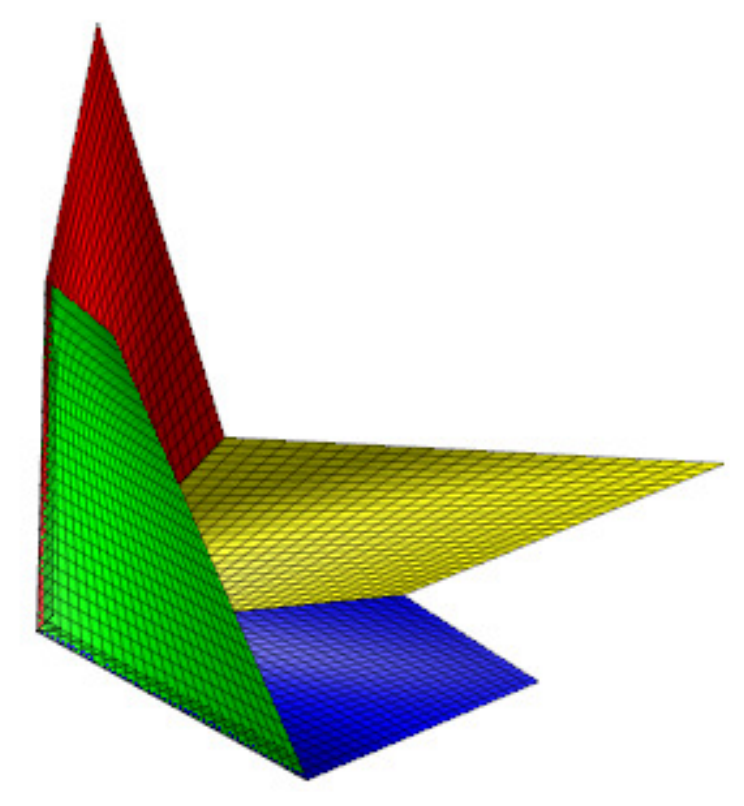}
\caption{\small Same cases as in Figure \ref{Fig-cuspidaledge-config-uu} with bilinear patches to help visualization.}\label{Fig:DiscreteCuspidalEdgesInterpolator}
\end{figure}

\subsection{Swallowtail vertices of the $q$-net}

The following proposition is a straightforward corollary of Proposition \ref{Prop-dmptl-construction}.

\begin{prop}\label{CuspMPTL}
Consider a vertex $x(u,v)$ of the DMPTL. The following conditions are equivalent:

\begin{enumerate}
\item
Two adjacent edges of the $x$-net with different labels are singular and the star is atypical.
\item
The two adjacent vertices of the DMPTL are in the same half-plane determined by the line supporting the chord $\alpha(u)\beta(v)$.
\end{enumerate}
\end{prop}

We say that $x(u,v)$ is a \emph{cusp} of the DMPTL if it satisfies one, and hence both, of the conditions of Proposition \ref{CuspMPTL}. To justify this definition, one should compare condition (2) of this proposition with the condition for a cusp in a smooth MPTL described in Section \ref{sec:CuspsMPTL}.

The following definition is central in the paper:

\begin{defn}
The vertex $q(u,v)$ of the asymptotic net is called a \emph{swallowtail} if two adjacent edges with different labels are singular and the star is atypical.
\end{defn}

From this definition, the next proposition is immediate:

\begin{prop}
A vertex $q(u,v)$ of the $q$-net is a swallowtail if and only if the corresponding vertex $x(u,v)$ of the $x$-net is a cusp of the DMPTL.
\end{prop}

We observe that the definition of a swallowtail vertex can be extended to any asymptotic net, even if it does not correspond to a DIIAS. In Figures \ref{Theo-swallowtail} and \ref{Fig:DiscreteSwallowtailInterpolator} we can see a swallowtail vertex $q(u,v)$. Observe the visual similarity of the smooth swallowtail and the discrete swallowtail with bilinear interpolators.

\begin{figure}[!htb]
\centering
\includegraphics[width=.9\linewidth]{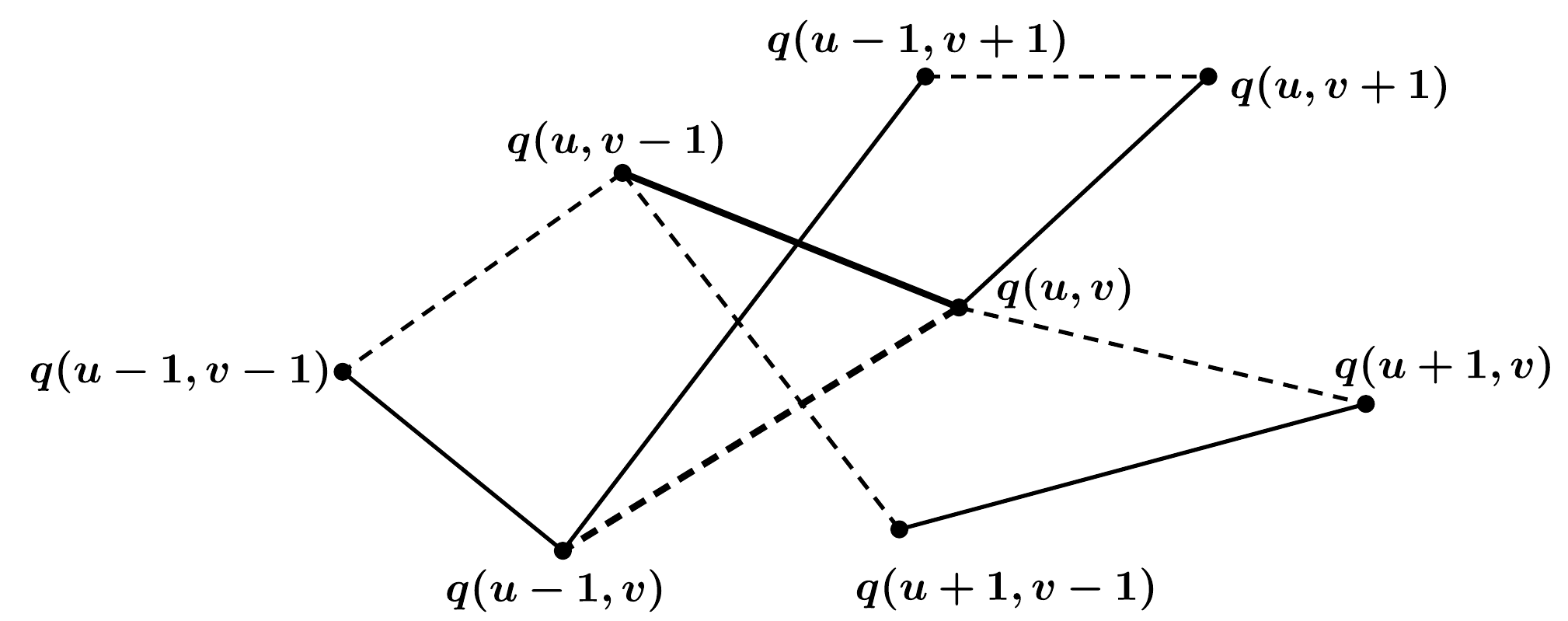}
\caption{\small Swallowtail at $q(u,v)$ with cuspidal edges $q_1(u-\tfrac{1}{2},v)$ (strongest dashed segment) and $q_2(u,v-\tfrac{1}{2})$ (strongest solid segment).}\label{Theo-swallowtail}
\end{figure}

\begin{figure}[!htb]
\centering
 \includegraphics[width=.50\linewidth]{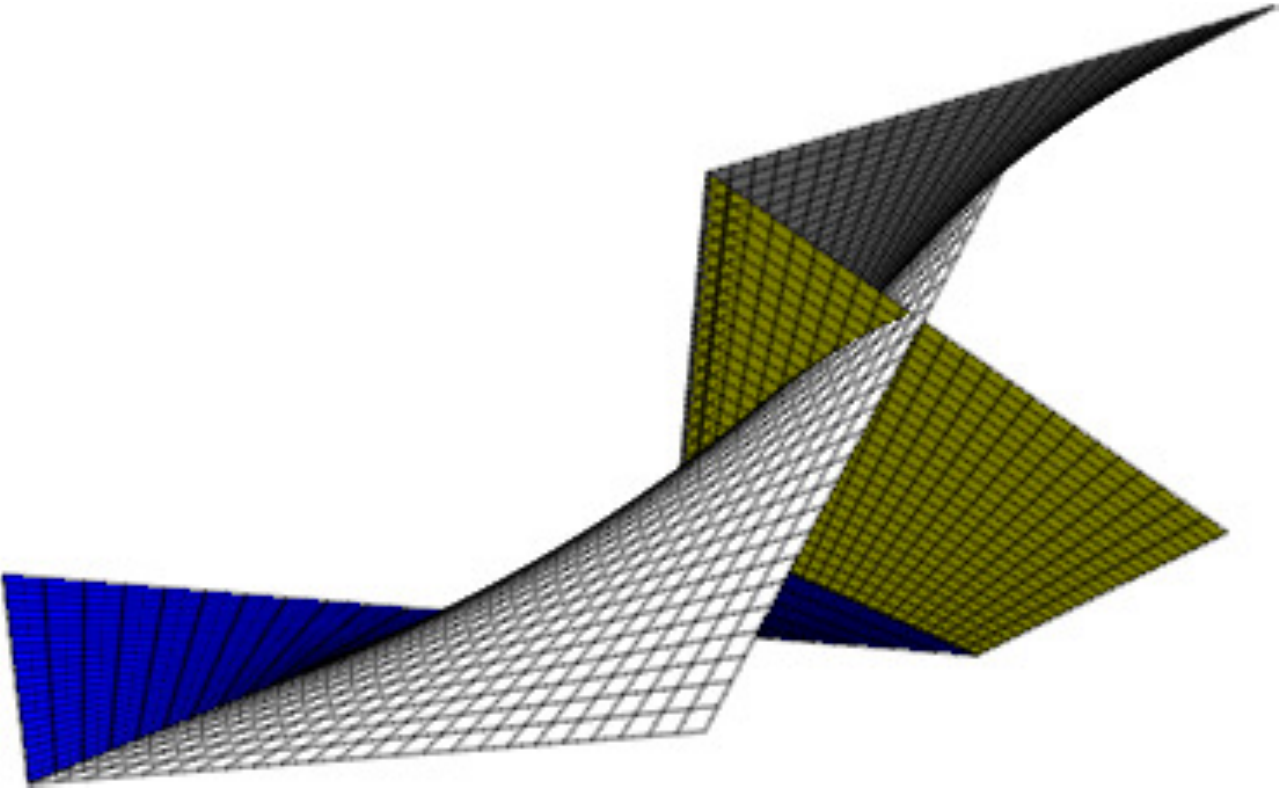}
\caption{\small A discrete swallowtail vertex with four bilinear patches distinguished by colors to help visualization.}\label{Fig:DiscreteSwallowtailInterpolator}
\end{figure}

\subsection{A swallowtail geometrical property}
Let us remember that a swallowtail vertex in the smooth case always implies in self-intersection, so we expect the same behavior in discrete cases too.

\begin{lem}\label{ModelNet}
Consider a DIIAS with $q(0,0)=(0,0,0)$,
$$
q(1,0)=(0,\beta,0),\ q(-1,0)=(\alpha,0,0),\ q(0,1)=(a,b,0),\ q(0,-1)=(c,d,0).
$$
Then
$$
\frac{z(-1,1)}{z(1,1)}=-\frac{\alpha\cdot b}{\beta\cdot a}, \ \ \frac{z(-1,-1)}{z(1,1)}=-\frac{\alpha\cdot d}{\beta\cdot a}, \ \ \frac{z(1,-1)}{z(1,1)}=\frac{c}{a}.
$$
\end{lem}
\begin{proof}
Observe first that
$$
x(1,1)=(a,b+\beta),\ x(-1,1)=(a+\alpha,b),\ x(-1,-1)=(c+\alpha,d),\ x(1,-1)=(c,d+\beta).
$$
Then the planarity of the stars at $(1,0)$, $(0,1)$, $(-1,0)$ and $(0,-1)$ implies the result.
\end{proof}

\begin{prop}
Let $q:I\times J\subset\Z^2\To\R^3$ be a DIIAS. If $q(u,v)$ is a swallowtail vertex, then there is a pair of quadrangles, with $q(u,v)$ as a vertex, whose corresponding bilinear patches intersect each other.
\end{prop}

\noindent\emph{Proof.}
Assume $q(0,0)=(0,0,0)$ is a swallowtail of the DIIAS, with adjacent cuspidal edges $q_1(-\tfrac{1}{2},0)$ and $q_2(0,\tfrac{1}{2})$. Taking into account the notation of Lemma \ref{ModelNet}, $a<0$, $b>0$, $c>0$ and $d>0$, which imply that $z(1,1)$ and $z(-1,1)$ have the same sign. We shall assume that both are positive, the other case being analogous.

Under the above assumptions, the bilinear patches $BP(\tfrac{1}{2},\tfrac{1}{2})$ and $BP(-\tfrac{1}{2},-\tfrac{1}{2})$ are both contained in the half-space $z\geq 0$. Moreover, the segment $q_1(-\tfrac{1}{2},0)\subset BP(-\tfrac{1}{2},-\tfrac{1}{2})$ is contained in the plane $z=0$ and is below $BP(\tfrac{1}{2},\tfrac{1}{2})$. Similarly, $q_2(0,\tfrac{1}{2})\subset BP(\tfrac{1}{2},\tfrac{1}{2})$ is contained in the plane $z=0$ and is below $BP(-\tfrac{1}{2},-\tfrac{1}{2})$. We conclude that necessarily $BP(\tfrac{1}{2},\tfrac{1}{2})\cap BP(-\tfrac{1}{2},-\tfrac{1}{2})\neq \emptyset$.
\hfill$\square$

\subsection{Example}\label{Example}

Consider the curves
$$
\alpha(u)=\left(u,\,5-\frac{(u-2)^2}{8}\right),\ u\in I,\ \ \ \beta(v)=\left(v^2-2,\,v\right),\ v\in J.
$$
By considering $I,J\subset\mathbb{R}$, we obtain a smooth IIAS by the centre-chord construction, and by considering $I,J\subset\mathbb{Z}$, we obtain a DIIAS by the discrete centre-chord construction.

\begin{figure}[!htb]
  \includegraphics[width=.4\linewidth]{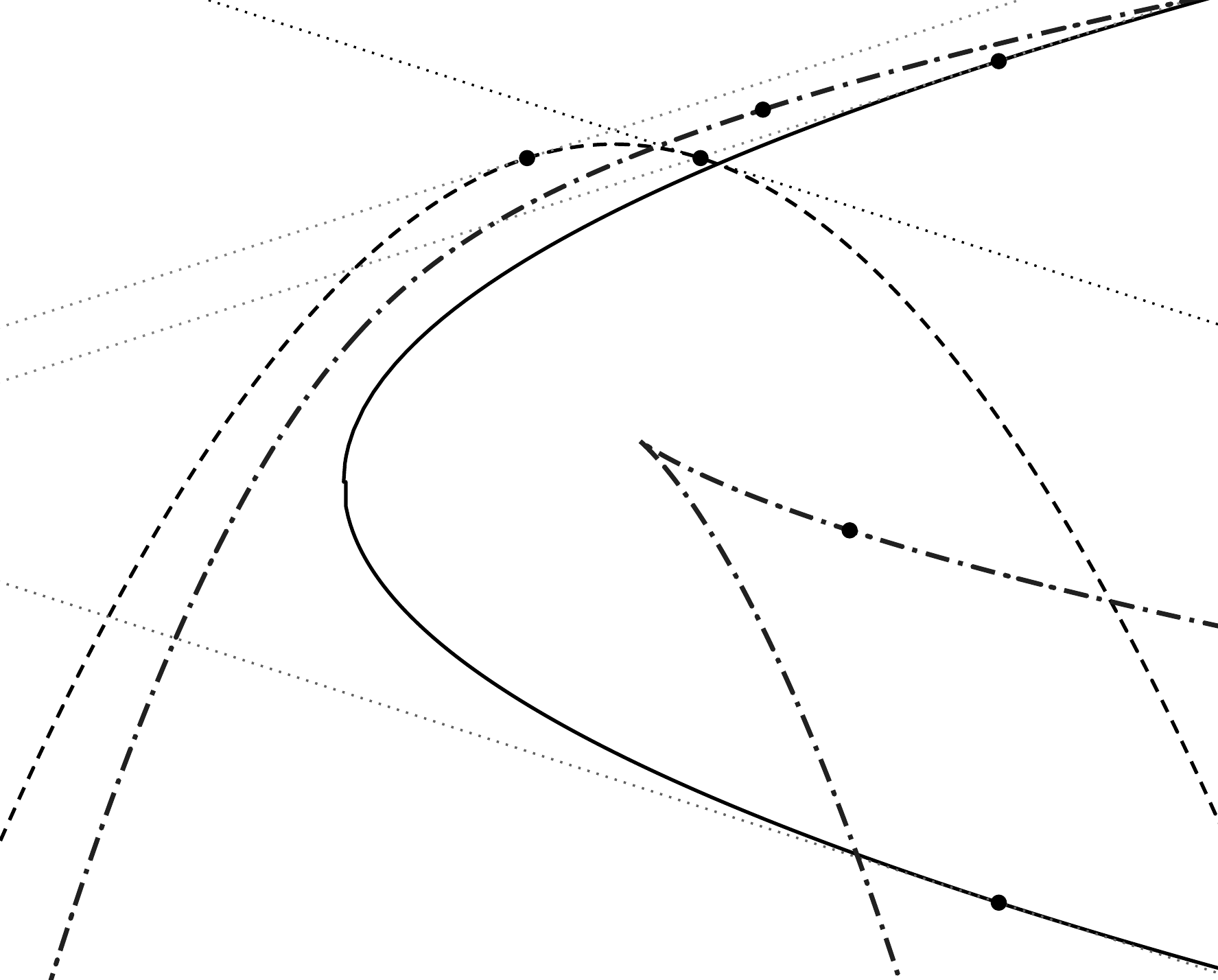}
 \hfill
 \includegraphics[width=.47\linewidth]{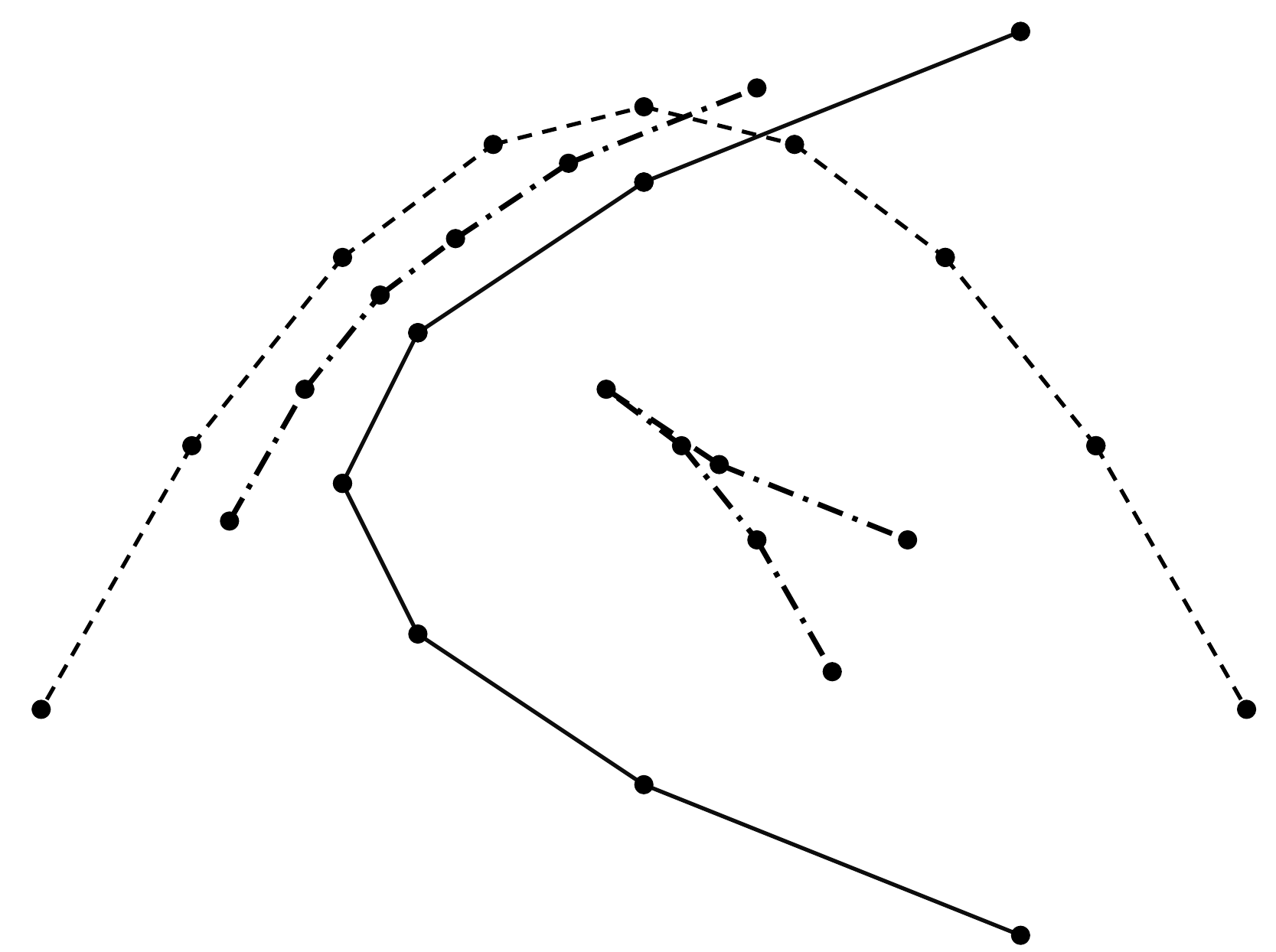}
\caption{\small MPTL (left) and DMPTL (right) associated to the pair of polygonal lines $(\alpha,\beta)$ of the example of Section \ref{Example}.}\label{Fig-Ex1-caustica-suave}
\end{figure}

We show in Figure \ref{Fig-Ex1-caustica-suave} the MPTL in the smooth case and the DMPTL in the discrete case. Note that both of them are formed by two connected components and only one of them presents a cusp, which means that the cuspidal curves of both surfaces (smooth and discrete) generated by the pair $(\alpha,\beta)$ have two connected components and a unique swallowtail vertex, as it can be seen in Figure \ref{Fig-ex-sing}.

\begin{figure}[!htb]
 \includegraphics[width=.48\linewidth]{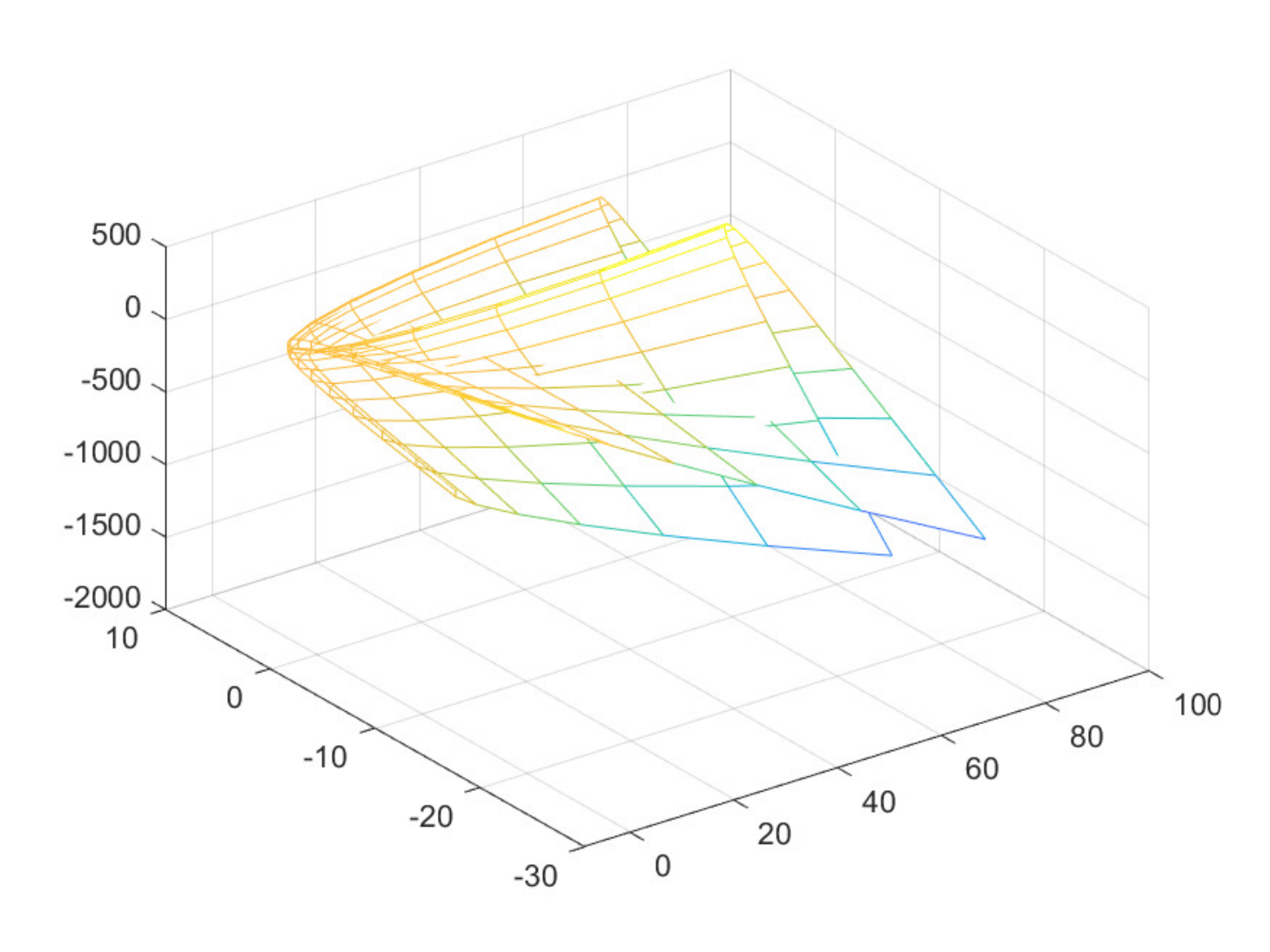}
 \hfill
 \includegraphics[width=.48\linewidth]{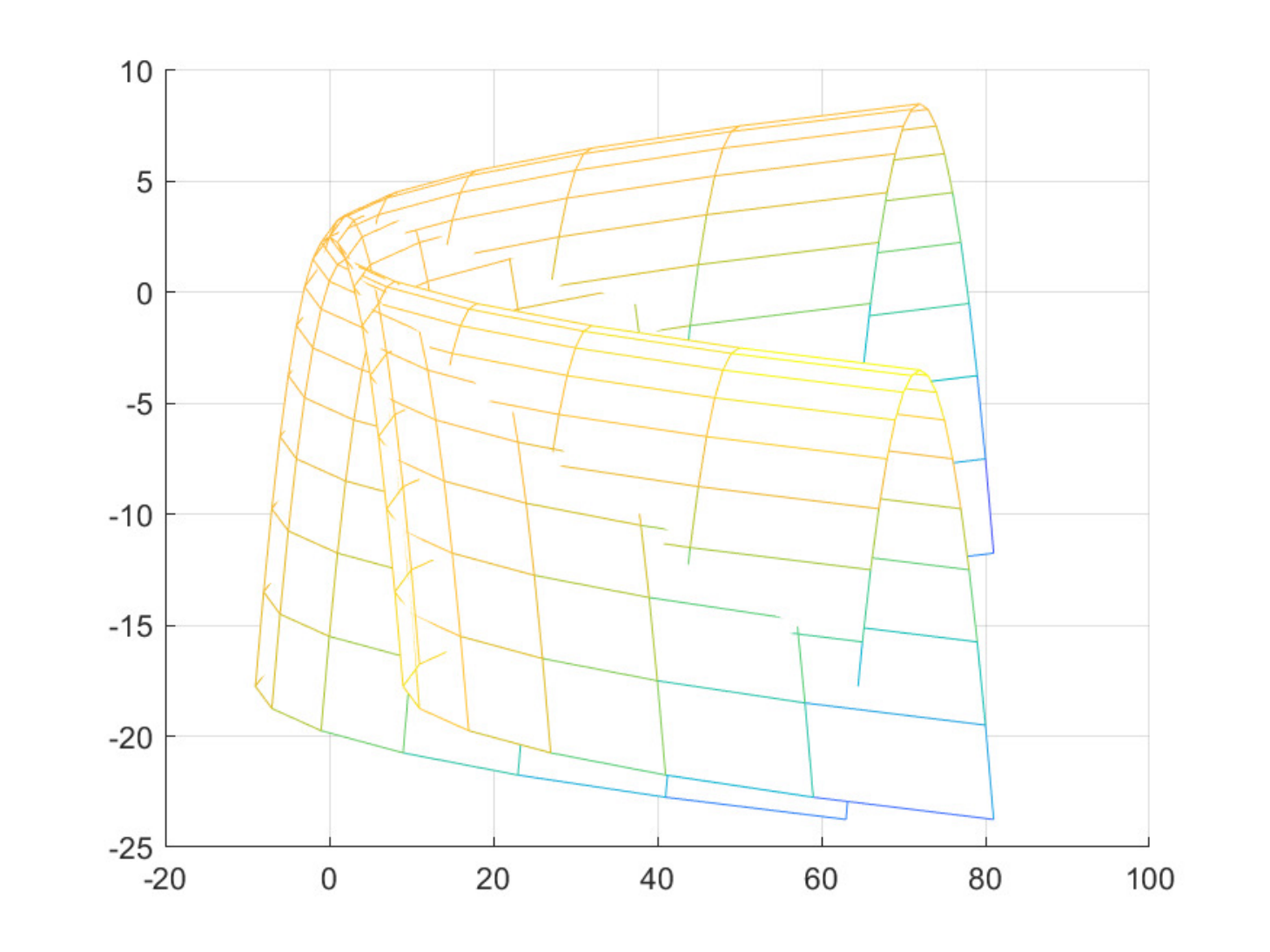}
\caption{\small Two views of the DIIAS of the example of Section \ref{Example}.}\label{Fig-ex-sing}
\end{figure}

\section{Ruled nets}

Ruled nets are defined in the same way as in smooth case, that is, in at least one of the coordinates direction, $u$-curves or $v$-curves are all straight lines. From Equations \eqref{eq-structural1}, one can easily check that a DIIAS is ruled if and only if $A=0$ or $B=0$.

\subsection{A characterization of ruled DIIAS}

Consider a ruled DIIAS. We may assume, w.l.o.g., that $B(v)=0$ and $\beta(v)=(0,v)$, $v\in J\subset\mathbb{Z}$. Next proposition is a discrete counterpart of Theorem \ref{teo-ruled-improper-affine-sphere}.

\begin{prop}
  Every ruled DIIAS without singularities is of the form $z=x^1x^2+\varphi(x^1)$, for some real function $\varphi$.
\end{prop}

\noindent \emph{Proof.}
Write $\alpha(u)=(\alpha^1(u),\alpha^2(u))$, $u\in I\subset\mathbb{Z}$. The hypothesis of no singularities implies that $\alpha^1_1(u+\tfrac{1}{2})$ does not change sign. Thus $\alpha^1(u)$ is an invertible map.

We have
$$
x(u,v)=(\alpha^1(u),\alpha^2(u)+v),\ y(u,v)=(-\alpha^1(u),v-\alpha^2(u)),
$$
$$
z_1(u+\tfrac{1}{2},v)=v\alpha^1_1(u+\tfrac{1}{2})-[\alpha_1(u+\tfrac{1}{2}),\alpha(u)],\ z_2(u,v+\tfrac{1}{2})=\alpha^1(u).
$$

\smallskip\noindent
By discrete integration on $u$ we have $z(u,v)=v\alpha^1(u)+g(u)$, where $g_1(u+\tfrac{1}{2})=-[\alpha_1(u+\tfrac{1}{2}),\alpha(u)]$ and so
$z=x^1x^2+\varphi(u)$,
where $\varphi(u)=g(u)-\alpha^1(u)\alpha^2(u)$. Since $x^1(u)=\alpha^1(u)$ is invertible, the proposition is proved. \hfill$\square$

\subsection{Discrete Cayley surface}

Let us now take a look at an example of a ruled DIIAS, known as discrete Cayley surface. Its discrete structure equations are
$$\begin{array}{l}
    q_{11}(u,v)=aq_2(u,v+\tfrac{1}{2})=a(0,1,u), \\
    q_{22}(u,v)=(0,0,0),\\
    q_{12}(u+\tfrac{1}{2},v+\tfrac{1}{2})=(0,0,1).
  \end{array}$$
We shall assume as initial conditions $q(0,0)=(0,0,0)$, $q(0,1)=(0,1,0)$ and $q(1,0)=(1,0,0)$. So the solution shall be
\begin{equation}\label{Cayley-surface-discrete}
  q(u,v)=\left(u, v+\frac{au(u-1)}{2},uv+\dfrac{au(u^2-1)}{6}\right), \;(u,v)\in\Z^2.
\end{equation}

Note that for this example, $\Omega=1$ and $q_{22}=(0,0,0)$. As in the smooth case, we shall call {\it normalized} a DIIAS satisfying these conditions. For a normalized DIIAS, the structural equations become
\begin{equation}
q_{11}(u,v)=A(u)q_2(u,v+\tfrac{1}{2}),\ \ q_{12}(u+\tfrac{1}{2},v+\tfrac{1}{2})=(0,0,1)\ \ q_{22}(u,v)=(0,0,0).
\end{equation}

Thus we have proved the following theorem:

\begin{thm}\label{Cayley-surface-characterization-discrete}
Let $q$ be a normalized DIIAS. Then $q$ is affinely congruent to a discrete Cayley surface if and only if $A\neq 0$ and $A_1=0$.
\end{thm}

This result should be considered as a discrete analog of Theorem \ref{Cayley-surface-characterization}.

\subsection{Singularities of ruled DIIAS}

Any ruled DIIAS can be otained by the centre-chord construction from a planar polygonal line $\alpha(u)$ and a planar straight line $\beta(v)$. We can observe that along any $v$-line with $u$ fixed, the corresponding bilinear patches are just extensions of each other (\cite{Vargas2023}).

The singular points of the DIIAS are the pairs $(u_0,v)$ so that $\tfrac{d\alpha}{du}(u_0)$ is parallel to $\beta$ and $v\in\mathbb{Z}$. Thus the DMPTL is generically a discrete set of lines parallel to $\beta$ and the singularities of the DIIAS are a discrete set of spatial lines, whose points are all cuspidal edges of the DIIAS.

\begin{example}\label{ex:Ruled}
Consider $D=\{-1,0,1\}\times\{-1,0,1\}$, i.e., $u\in\{-1,0,1\}$ and $v\in\{-1,0,1\}$. Define the polygonal $\alpha$ by
$\alpha(-1)=(-1,3)$, $\alpha(0)=(0,2)$ and $\alpha(1)=(1,5)$, and the polygonal line $\beta$ by $\beta(-1)=(-1,0)$,
$\beta(0)=(0,0)$ and $\beta(1)=(1,0)$. Observe that the $\beta$ segments are colinear, and so the DIiAS obtained by the center-chord construction is ruled. In Figure \ref{Fig:Ruled1} (center), we can see the polygonal lines $\alpha$ and $\beta$ together with the corresponding $x$-net. The edges of the DMPTL are $\left(u=0,v=\tfrac{1}{2}\right)$ and $\left(u=0,v=-\tfrac{1}{2}\right)$.

\begin{figure}[!htb]
\includegraphics[width=.50\linewidth]{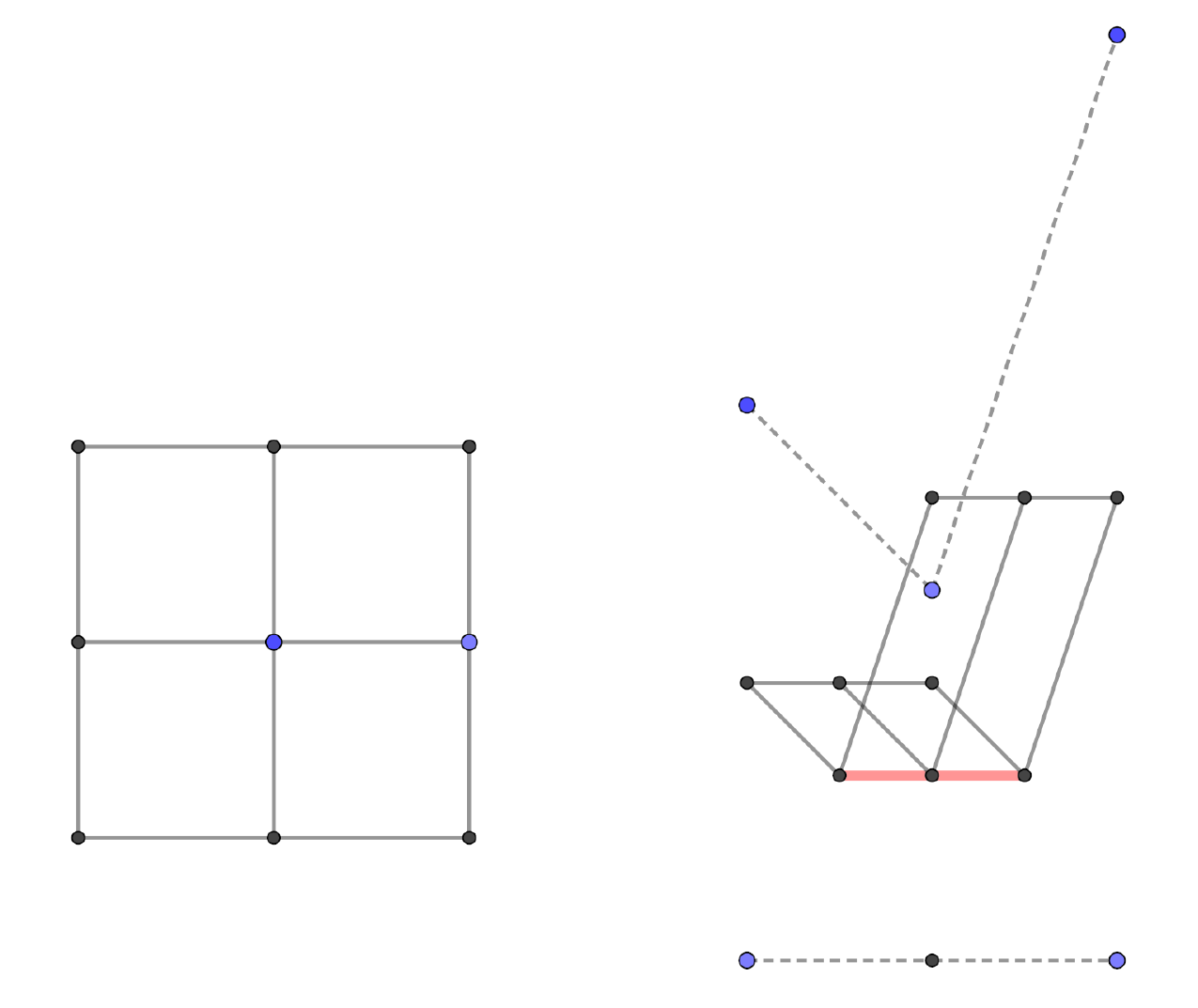}
\hfill
\includegraphics[width=.45\linewidth]{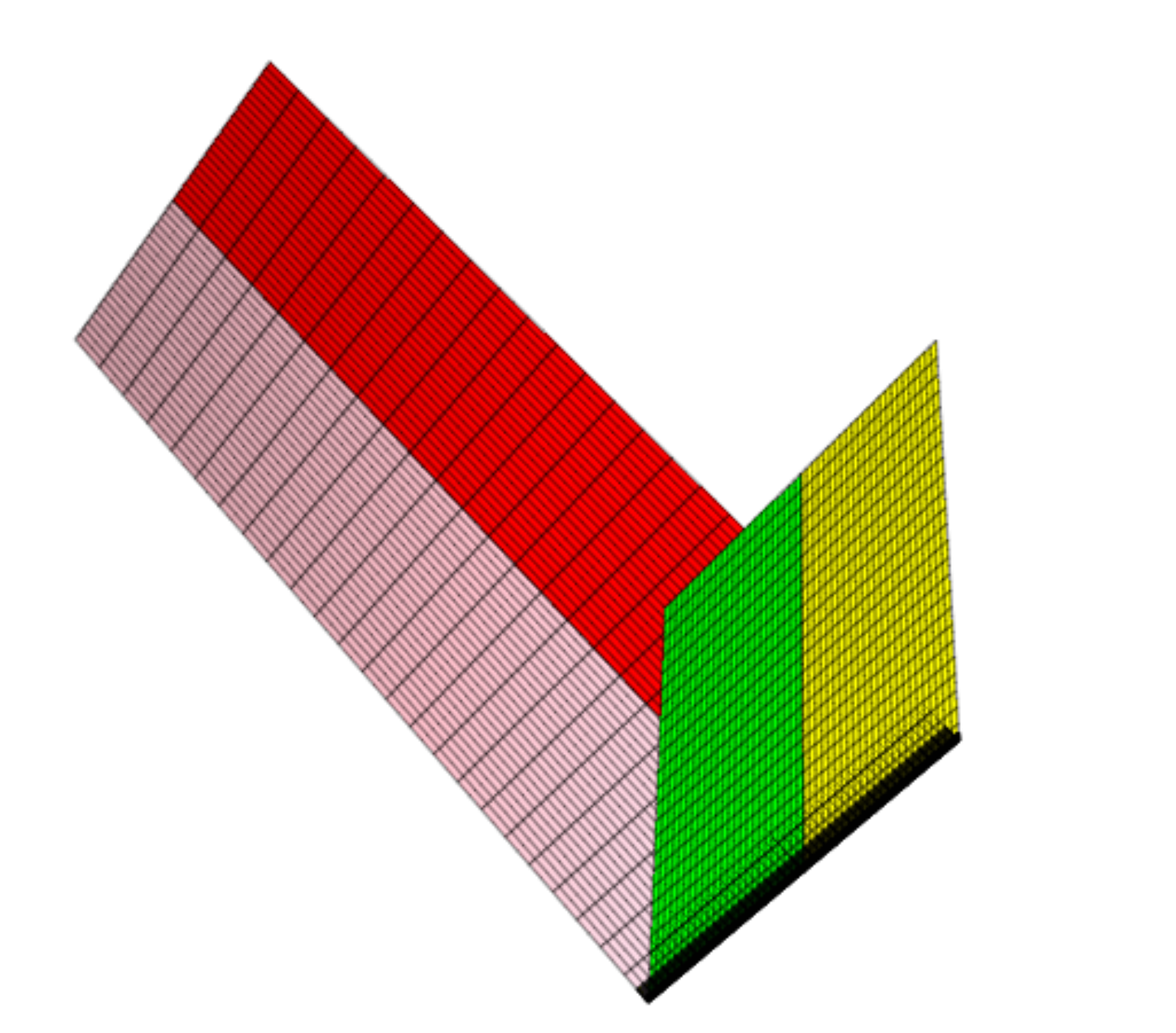}
\caption{\small The domain $D$ (left), the polygonal lines $\alpha$ and $\beta$ (traced center), the corresponding $x$-net (full segments center) and the four bilinear patches of the DIIAS (right) of Example \ref{ex:Ruled}.  The edges of the DMPTL of the pair $(\alpha,\beta)$ appear in thick red (center), while the cuspidal edges of the DIIAS appear in thick black (right).}
\label{Fig:Ruled1}
\end{figure}

To each quadrangle corresponds a bilinear patch. For the quadrangle $(-\tfrac{1}{2},-\tfrac{1}{2})$, the bilinear patch is
$$
BP(-\tfrac{1}{2},-\tfrac{1}{2})(u,v)=\tfrac{1}{2}\left( u+v, 2-u, -u-v+\tfrac{1}{2}uv  \right),\ \ -1\leq u\leq 0, \ \ -1\leq v\leq 0,
$$
while for the quadrangle $(-\tfrac{1}{2},\tfrac{1}{2})$, the bilinear patch is
$$
BP(-\tfrac{1}{2},\tfrac{1}{2})(u,v)=\tfrac{1}{2}\left( u+v, 2-u, -u-v+\tfrac{1}{2}uv  \right),\ \ -1\leq u\leq 0, \ \ 0\leq v\leq 1.
$$
Note that these two bilinear patches are extensions of each other, as expected. For the quadrangle $(\tfrac{1}{2},-\tfrac{1}{2})$, the bilinear patch is
$$
BP(\tfrac{1}{2},-\tfrac{1}{2})(u,v)=\tfrac{1}{2}\left( u+v, 2+3u, -u-v-\tfrac{3}{2}uv  \right),\ \ -1\leq u\leq 0, \ \ -1\leq v\leq 0,
$$
while for the quadrangle $(-\tfrac{1}{2},\tfrac{1}{2})$, the bilinear patch is
$$
BP(\tfrac{1}{2},\tfrac{1}{2})(u,v)=\tfrac{1}{2}\left( u+v, 2+3u, -u-v-\tfrac{3}{2}uv  \right),\ \ -1\leq u\leq 0, \ \ 0\leq v\leq 1.
$$
Note once again that these two bilinear patches are extensions of each other, as expected.
The edges above the DMPTL are cuspidal edges of the $q$-net, and in a ruled net they are colinear. In Figure \ref{Fig:Ruled1} (right), we can see the cuspidal edges in thick black.

\end{example}

\section*{\textbf{Acknowledgement}}

The authors are thankful to CAPES and CNPq for financial support during the preparation of this paper. Both authors thank Pontifical Catholic University of Rio de Janeiro and the first author also thanks Col\'egio Pedro II.

\subsection*{Funding:} Both authors had the support of CNPq (Conselho Nacional de Desenvolvimento Cient\'ifico e Teconol\'ogico - Brazil) and CAPES (Coordena\c c\~ao de Aperefei\c coamento de Pessoal de N\'ivel Superior - Brazil).

\subsection*{Author Contribution:} Both authors wrote the article, prepared the figures and reviewed the article.

\subsection*{Conflict of Interest:} The authors have no conflict of interest.

\subsection*{Data Availability Statement:} This article is theoretical and does not use experimental data.

\bibliographystyle{amsplain}

\begin{thebibliography}{99}
\bibitem{Bobenko1999} BOBENKO, A., SCHIEF, W.: \textbf{Affine spheres: Discretization via duality relations}. Experimental Mathematics \textbf{8}(3), 261-280 (1999).
\bibitem{Bobenko2008} BOBENKO, A., SURIS, Y.: \textbf{Discrete Differential Geometry: Integrable Structure}. Graduate Studies in Mathematics, Vol. \textbf{98}, AMS (2008).
\bibitem{Buchin1983} BUCHIN, S.: \textbf{Affine Differential Geometry}. Science Press, Beijing, China (1983).
\bibitem{Craizer2008} CRAIZER, M., SILVA, M., TEIXEIRA, R.: \textbf{Area distance of convex plane curves and improper affine spheres}. SIAM Journal on Imaging Sciences \textbf{1}(3), 209–227 (2008).
\bibitem{Craizer2010} CRAIZER, M., ANCIAUX, H., LEWINER, T.: \textbf{Discrete Affine Minimal Surfaces with Indefinite Metric}. Differential Geometry and its Applications \textbf{28}, 158-169 (2010).
\bibitem{Craizer2011} CRAIZER, M., SILVA, M., TEIXEIRA, R.: \textbf{A Geometric Representation of Improper Indefinite Affine Sphere with Singularities}. Journal of Geometry \textbf{100}, 65-78 (2011).
\bibitem{Giblin2008} GIBLIN, P.J.: \textbf{Affinely Invariant Symmetry Sets}. In: \textbf{Geometry and Topology of Caustics}. Vol. 82, 71-84. Banach Center Publications (2008).
\bibitem{Hoffmann2012} HOFFMANN, T., ROSSMAN, W., SASAKI, T., and YOSHIDA, M.: \textbf{Discrete flat surfaces and linear Weingarten surfaces in hyperbolic 3-space}. Trans. Amer. Math. Soc., \textbf{364}(11), 5605-5644, (2012).
\bibitem{Rorig2014} HUHNEN-VENEDEY, E., RÖRIG, T.: \textbf{Discretization of Asymptotic Line Parametrizations using Hyperboloid Surface Patches}. Geometriae Dedicata \textbf{168}, 265-289 (2014).
\bibitem{Ishikawa2006} ISHIKAWA, G. MACHIDA, Y.: \textbf{Singularities of Improper Affine Spheres and Surfaces of Constant Gaussian Curvature}. International Journal of Mathematics \textbf{17}(3), 269-293 (2006).
\bibitem{Kaferbock2013} KÄFERBÖCK, F., POTTMANN, H.: \textbf{Smooth Surfaces from Bilinear Patches: Discrete Affine Minimal Surfaces}. Computer Aided Geometric Design \textbf{30}, 476-489 (2013).
\bibitem{Kobayashi2020} KOBAYASHI, S., MATSUURA, N.: \textbf{Representation Formula for Discrete Indefinite Affine Spheres}. Differential Geometry and its Applications \textbf{69}, 1-41 (2020).
\bibitem{Matsuura2003} MATSUURA, N., URAKAWA, H.: \textbf{Discrete improper affine spheres}. Journal of Geometry and Physics \textbf{45}, 164-183 (2003).
\bibitem{Milan2014} MILÁN, F.: \textbf{The Cauchy Problem for Indefinite Improper Affine Spheres and their Hessian Equation}. Advances in Mathematics \textbf{251}, 22-34 (2014).
\bibitem{Nomizu1994} NOMIZU, K., SASAKI, T.: \textbf{Affine Differential Geometry}. Cambridge University Press (1994).
\bibitem{Rossman2018} ROSSMAN, W., YASUMOTO, M.: \textbf{Discrete linear Weingarten surfaces with singularities in Riemannian and Lorentzian spaceforms}. Advanced Studies in Pure Mathematics, \textbf{78}, 383-410 (2018).
\bibitem{Vargas2023} VARGAS, A.R., CRAIZER M.: \textbf{Discrete asymptotic nets with constant affine mean curvature}, Beitr. Algebra Geom., https://doi.org/10.1007/s13366-023-00707-w, 2023.
\bibitem{Yasumoto2015} YASUMOTO, M.: \textbf{Discrete maximal surfaces with singularities in Minkowski space}. Differential Geometry and its Applications, \textbf{43}, 130-154 (2015).
\end{thebibliography}

\end{document}